\documentclass[11pt]{article}
\usepackage[vcentermath,enableskew]{youngtab}
\usepackage{amssymb}
\usepackage{amsthm}
\usepackage{amsmath}
\usepackage{array}
\usepackage{longtable}
\allowdisplaybreaks

\newtheorem{theorem}{Theorem}[section]
\newtheorem{lemma}{Lemma}[section]
\newtheorem{proposition}{Proposition}[section]
\newtheorem{corollary}{Corollary}[section]

\numberwithin{equation}{section}

\newcommand{\ep}{\varepsilon}
\newcommand{\epi}{\varepsilon_{\iota}}

\newcommand{\Om}{\Omega}

\newcommand{\bW}{\boldsymbol{W}}

\newcommand{\FF}{\mathbb{F}}

\newcommand{\SO}{\mathrm{SO}}
\newcommand{\PSO}{\mathrm{PSO}}

\newcommand{\PSp}{\mathrm{PSp}}
\newcommand{\POm}{\mathrm{P}\Omega}
\newcommand{\Sp}{\mathrm{Sp}}
\newcommand{\SL}{\mathrm{SL}}
\newcommand{\SU}{\mathrm{SU}}
\newcommand{\gO}{\mathrm{O}}

\newcommand{\Irr}{\mathrm{Irr}}

\def\adots{\mathinner{\mkern2mu\raise0pt\hbox{.}  
\mkern2mu\raise4pt\hbox{.}\mkern1mu
\raise7pt\vbox{\kern7pt\hbox{.}}\mkern1mu}}

\begin{document}

\bibliographystyle{amsplain}

\title{On involutions and indicators of finite orthogonal groups}
\author{Gregory K. Taylor and C. Ryan Vinroot}
\date{}

\maketitle

\begin{abstract}
We study the numbers of involutions and their relation to Frobenius-Schur indicators in the groups $\SO^{\pm}(n,q)$ and $\Om^{\pm}(n,q)$.  Our point of view for this study comes from two motivations.  The first is the conjecture that a finite simple group $G$ is strongly real (all elements are conjugate to their inverses by an involution) if and only if it is totally orthogonal (all Frobenius-Schur indicators are 1), and we are able to show this holds for all finite simple groups $G$ other than the groups $\Sp(2n,q)$ with $q$ even or $\Om^{\pm}(4m,q)$ with $q$ even.  We prove computationally that for small $n$ and $m$ this statement indeed holds for these groups by equating their character degree sums to the number of involutions.  We also prove a result on a certain twisted indicator for the groups $\SO^{\pm}(4m+2,q)$ with $q$ odd.   Our second motivation is to continue the work of Fulman, Guralnick, and Stanton on generating function and asymptotics for involutions in classical groups.  We extend their work by finding generating functions for the numbers of involutions in $\SO^{\pm}(n,q)$ and $\Om^{\pm}(n,q)$ for all $q$, and we use these to compute the asymptotic behavior for the number of involutions in these groups when $q$ is fixed and $n$ grows.
\end{abstract}

\section{Introduction}

An element $g$ in a group $G$ is \emph{real} in $G$ if $g$ is conjugate to its inverse in $G$.  In this case, all elements in the $G$-conjugacy class of $g$ are also real, and the class is called a \emph{real class}.  If $g$ is $G$-conjugate to its inverse by an element $h \in G$ which satisfies $h^2 = 1$, then we say that the element $g$ is \emph{strongly real} in $G$, and the conjugacy class of $g$ in $G$ is called a \emph{strongly real class}.    If all classes of $G$ are real (or strongly real), then $G$ is said to be a \emph{real group} (or a \emph{strongly real group}).  If $G$ is a finite group, then the number of real classes of $G$ is equal to the number of complex irreducible characters of $G$ which are real-valued.  If $\chi$ is a real-valued irreducible character of $G$, then the complex irreducible representation $(\pi, V)$ which affords $\chi$ may or may not be realizable as a representation over the real numbers.  The \emph{Frobenius-Schur indicator} of $\chi$, denoted $\ep(\chi)$, is defined as $\ep(\chi) = 1$ if $(\pi, V)$ is realizable as a real representation, $\ep(\chi) = -1$ if $(\pi, V)$ is not realizable as a real representation but $\chi$ is real-valued, and $\ep(\chi)=0$ if $\chi$ is not real-valued.  The group $G$ is said to be \emph{totally orthogonal} if $\ep(\chi)=1$ for all complex irreducible characters $\chi$ of $G$.  

Brauer \cite[Problem 14]{Br63} asked for a group-theoretical interpretation for the number of complex irreducible $\chi$ of $G$ which satisfy $\ep(\chi)=1$.  Because of many important classes of examples, it is suspected that an answer to Brauer's question should involve the strongly real classes of $G$.  In particular, it has been stated as a conjecture in \cite{KaKu15} (attributed to P. H. Tiep) that if $G$ is a finite simple group, then $G$ is strongly real if and only if $G$ is totally orthogonal.  It is known that this does not hold for finite groups in general; indeed, Kaur and Kulshretha \cite{KaKu15} have exhibited families of $2$-groups groups that are strongly real but not totally orthogonal, and families which are totally orthogonal but not strongly real.

After establishing some preliminary notation and results in Section \ref{Pre}, we consider in Section \ref{ReSimp} the conjecture that a finite simple group is strongly real if and only if it is totally orthogonal.  Except for the cases that $G$ is of the form $\Omega^{\pm}(4m, q)$ with $q$ even or $\Sp(2n,q)$ with $q$ even, we observe that this conjecture holds for all other finite simple groups $G$ by applying the classification of strongly real simple groups and known results on Frobenius-Schur indicators, in Theorem \ref{simple}.  In Section \ref{NumInvols}, we enumerate various sets of involutions in the finite orthogonal groups and certain subgroups, which is directly related to the above conjecture since a finite group $G$ is totally orthogonal if and only if its character degree sum is equal to the number of involutions in $G$ (see Lemma \ref{Allepi1}).  We immediately apply these counts in Section \ref{ResIndicators}, where we use data obtained by L\"{u}beck \cite{LuWWW} to show computationally that $\Omega^{\pm}(4m,q)$ and $\Sp(2n,q)$ with $q$ even are indeed totally orthogonal for small rank, and that certain twisted indicators of $\Omega^{\pm}(4m+2,q)$ are all $1$ for small rank, in Theorems \ref{Spqeven} and \ref{OOmqeven}.  We show the analogous twisted indicators are always $1$ for $\SO^{\pm}(4m+2,q)$ when $q$ is odd in Theorem \ref{SOtwist}.

In Sections \ref{GenFuns} and \ref{Asymptotics} we extend results of Fulman, Guralnick, and Stanton \cite{FuGuSt16} on the asymptotics for the numbers of involutions in finite classical groups.  Specifically, in Section \ref{GenFuns} we find generating functions for the numbers of involutions in $\SO^{\pm}(n,q)$ and $\Om^{\pm}(n,q)$ (for $q$ odd and even), and in Section \ref{Asymptotics} we compute asymptotics for the number of involutions in these groups (as $n$ grows and $q$ is fixed), as is suggested could be done in the introduction of \cite{FuGuSt16}.  Several interesting results emerge, for example in Corollary \ref{onehalf} we see that as $n$ grows $\Om^{\pm}(n,q)$ asymptotically has half the number of involutions as $\SO^{\pm}(n,q)$, when $q$ is odd and fixed.  The generating functions in Section \ref{GenFuns} are also significant in the problem of computing Frobenius-Schur indicators for the groups of interest, as several of them are also generating functions for character degree sums.  These should be useful in extending techniques from \cite{FuVi14} to other groups. \\
\\
{\bf Acknowledgements. }  Vinroot was supported in part by a grant from the Simons Foundation, Award \#280496. 

\section{Preliminaries} \label{Pre}

\subsection{The finite orthogonal groups}

We primarily study various types of orthogonal groups defined over finite fields, which we now discuss.  Considerations of other groups are isolated to Section \ref{ReSimp}.  

Let $\FF_q$ denote a finite field with $q$ elements, and let $V$ be a vector space over $\FF_q$ of dimension $N$ where $N \in \{ 2n, 2n+1 \}$.  We first consider the case that $q$ is odd.  Let $Q$ be a non-degenerate quadratic form defined on $V$ (that is, the symmetric form on $V$ defined by $Q$ is non-degenerate).  In this case, $Q$ must be equivalent to one of four possible forms.  We label these as follows, where the form is given by the image of a vector in $V$ with some coordinates $(x_1, \ldots, x_N)$:
\[
\begin{array}{lcc}
(\mathbf{0})  &  ~   & x_1x_2 + x_3x_4 + \cdots + x_{2n-1}x_{2n} \\
(\mathbf{w})  &  ~   & x_1x_2 + x_3x_4 + \cdots + x_{2n-3}x_{2n-2} + x_{2n-1}^2 - \delta x_{2n}^2 \\
(\mathbf{1})  &  ~   & x_1x_2 + x_3x_4 + \cdots + x_{2n-1}x_{2n} + x_{2n+1}^2 \\
(\mathbf{d})  &  ~   & x_1x_2 + x_3x_4 + \cdots + x_{2n-1}x_{2n} + \delta x_{2n+1}^2,
\end{array}
\]
where $\delta$ is some fixed element in $\FF_q^{\times} \setminus \FF_q^{\times 2}$.  We define $\bW = \{ \mathbf{0}, \mathbf{w}, \mathbf{1}, \mathbf{d} \}$, and the {\it Witt type} of $Q$ is defined to be the symbol in $\bW$ corresponding to the quadratic form in the list above to which $Q$ is equivalent.  We may define the orthogonal group on $V$ corresponding to $Q$, denoted by $\gO_{Q}(V)$, and since the isomorphism type depends only on the Witt type, we may also write this as $\gO_{Q}(V) = \gO^{(\mathbf{s})}(N,q)$, where $\mathbf{s} \in \bW$ is the Witt type of $Q$.  When $N = 2n$ is even, we will denote $\gO^{(\mathbf{0})}(N,q) = \gO^+(2n,q)$ and $\gO^{(\mathbf{w})}(N,q) = \gO^-(2n,q)$.  When $N = 2n+1$ is odd, denote $\gO^{(\mathbf{1})}(N,q) = \gO^+(2n+1,q)$ and $\gO^{(\mathbf{d})}(N,q) = \gO^-(2n+1,q)$, and in fact we further have 
$$\gO^+(2n+1, q) \cong \gO^-(2n+1, q),$$ which we commonly denote by $\gO(2n+1,q)$.  We will write $\gO^{\pm}(N,q)$ for a general orthogonal group of one of the types just described.  We let $\SO^{\pm}(N,q)$ denote the special orthogonal group, or the index 2 subgroup consisting of determinant $1$ elements of the corresponding orthogonal group $\gO^{\pm}(N,q)$.  Note that we have
$$ \gO(2n+1, q) \cong \SO(2n+1,q) \times \{ \pm 1 \},$$
where $ Z = \{ \pm 1 \}$ is the center of $\gO^{\pm}(N,q)$.  We let $\Omega^{\pm}(N,q)$ denote the derived subgroup of $\SO^{\pm}(N,q)$, which has index 2 in $\SO^{\pm}(N,q)$.  Then $\Omega^{\pm}(N,q)$ has center $Z = \Omega^{\pm}(N,q) \cap \{1, -1\}$, and the projective group $\POm^{\pm}(N,q) = \Omega^{\pm}(N,q)/Z$ is known to be simple in all but finitely many cases (see \cite[Theorem 6.31]{Gr02}, for example).

We now consider the case that $q$ is even, and the orthogonal group $\gO_Q(V)$ corresponding to some {\it non-defective} quadratic form $Q$ defined on $V$.  The classification of equivalence classes of such quadratic forms when $q$ is even is only somewhat different from the situation when $q$ is odd given above, although we do not need the full description here, and we refer to \cite[Chapter 12]{Gr02} for details.  In particular, when $\dim(V)=2n+1$ is odd, there is again only one isomorphism class of orthogonal groups, which we denote by $\gO(2n+1, q)$.  In this case, we have 
$$\gO(2n+1, q) \cong \Sp(2n,q),$$ 
where $\Sp(2n,q)$ is the symplectic group corresponding to the unique class of non-degenerate alternating forms that one can define on an $\FF_q$-vector space of dimension $2n$, for example see \cite[Theorem 14.2]{Gr02}.  When $q$ is even, we have that $\Sp(2n,q) = \PSp(2n,q)$ is a simple group for all but finitely many cases.   When $N=2n$ is even, there are two equivalence classes of non-defective quadratic forms on $V$, see \cite[Theorem 12.9(2)]{Gr02}.  We write $\gO^+(2n,q)$ for the group corresponding to the form which may be defined as $Q(x_1, \ldots, x_{2n}) = x_1 x_{1 + n} + x_2 x_{2+n} + \cdots + x_n x_{2n}$, and we write $\gO^-(2n,q)$ for the group corresponding to the other possible class of quadratic form.  In the case that $q$ is even, we let $\Omega^{\pm}(2n,q)$ denote the derived group of $\gO^{\pm}(2n,q)$, which has index 2 in $\gO^{\pm}(2n,q)$, and is a simple group for all but finitely many cases.  There is also the following useful criterion for determining when an element of $\gO^{\pm}(2n,q)$ is in its derived subgroup $\Omega^{\pm}(2n,q)$ (see \cite[Proposition 3.2]{Ra11}).

\begin{lemma} \label{OmCritqeven}
If $q$ is even and $g \in \gO^{\pm}(2n,q)$, then $g \in \Om^{\pm}(2n,q)$ if and only if $\mathrm{rank}(1+g)$ is even.
\end{lemma}

We will need to use the orders of the groups described above numerous times throughout this paper, and so we list them here for reference.  The first formula holds whether $q$ is even or odd:
\begin{equation} \label{OrderOneven}
|\gO^{\pm}(2n,q)| = 2 q^{n(n-1)} (q^n \mp 1) \prod_{i=1}^{n-1} (q^{2i} -1).
\end{equation}
When $q$ is odd,
\begin{equation} \label{OrderOnoddqodd}
|\gO(2n+1, q)| = 2 q^{n^2} \prod_{i = 1}^n (q^{2i} - 1).
\end{equation}
When $q$ is even,
\begin{equation} \label{OrderOnoddqeven}
|\gO(2n+1, q)| = |\Sp(2n,q)| = q^{n^2} \prod_{i=1}^n (q^{2i} - 1).
\end{equation}
We also have
\begin{equation} \label{OrderSOOmqodd}
|\SO^{\pm}(N,q)| = \frac{1}{2} |\gO^{\pm}(N,q)| \quad \text{ and } \quad |\Om^{\pm}(N,q)| = \frac{1}{2} |\SO^{\pm}(N,q)| \quad \text{ when } q \text{ is odd,}
\end{equation}
and 
\begin{equation} \label{OrderOmqeven}
|\Om^{\pm}(2n,q)| = \frac{1}{2} |\gO^{\pm}(2n,q)| \quad \text{ when } q \text{ is even.}
\end{equation}
Note that when $q$ is odd we have $|\gO(1,q)| = 2$, while when $q$ is even we have $|\gO(1,q)| = |\Sp(0,q)| = 1$.  For $q$ even or odd, we define $|\gO^{+}(0,q)| =1$, but we do not define $|\gO^-(0,q)|$.

\subsection{Characters and Frobenius-Schur indicators}

Let $G$ be a finite group, and we let $\Irr(G)$ denote the collection of irreducible complex characters of $G$, and we let $\langle \cdot, \cdot \rangle$ denote the standard inner product on class functions of $G$.  If $\chi \in \Irr(G)$ and $(\pi, V)$ is the representation of $G$ affording $\chi$, let $[\pi]_{\mathcal{B}}$ denote the matrix representation corresponding to a choice of basis $\mathcal{B}$ for $V$.  We say that $(\pi,V)$ is \emph{defined over $\mathbb{R}$} if there is a basis $\mathcal{B}$ for $V$ such that $[\pi(g)]_{\mathcal{B}}$ has all entries in $\mathbb{R}$ for all $g \in G$.  Recall that the \emph{Frobenius-Schur indicator} of $\chi \in \Irr(G)$, denoted $\ep(\chi)$, is defined as $\ep(\chi) = 0$ if $\chi$ is not real-valued, $\ep(\chi) = 1$ if $(\pi, V)$ is defined over $\mathbb{R}$, and $\ep(\chi) = -1$ if $\chi$ is real-valued but $(\pi, V)$ is not defined over $\mathbb{R}$.  Then we have $\ep(\chi) = \frac{1}{|G|} \sum_{g \in G} \chi(g^2)$, and $\sum_{\chi \in \Irr(G)} \ep(\chi) \chi(1) = \# \{ g \in G \, \mid \, g^2 = 1\}$ (see \cite[Chapter 4]{Is76}, for example).

We now discuss a useful generalization of the Frobenius-Schur indicator.  Let $\iota$ be an automorphism of $G$ such that $\iota^2 = 1$, and we denote the action of $\iota$ on $g \in G$ by ${^\iota g}$.  We also let ${^\iota \chi}$ and ${^\iota \pi}$ be defined by ${^\iota \chi}(g) = \chi({^\iota g})$ and ${^\iota \pi}(g) = \pi({^\iota g})$, respectively.  Given $\chi \in \Irr(G)$, the representation $(\pi, V)$ which affords $\chi$, and the automorphism $\iota$ of $G$, define the \emph{twisted Frobenius-Schur indicator} of $\chi$, denoted $\epi(\chi)$, by
$$ \epi(\chi) = \left\{ \begin{array}{rl} 1 & \text{ if there is some } \mathcal{B} \text{ such that } [{^\iota \pi}(g)]_{\mathcal{B}} = \overline{[\pi(g)]_{\mathcal{B}}} \text{ for all } g \in G, \\ -1 & \text{ if } {^\iota \chi} = \bar{\chi} \text{ but there is no } \mathcal{B} \text{ such that } [{^\iota \pi}(g)]_{\mathcal{B}} = \overline{[\pi(g)]_{\mathcal{B}}} \text{ for all } g \in G, \\ 0 & \text{ if } {^\iota \chi} \neq \bar{\chi}. \end{array} \right. $$

Note that $\epi(\chi) = \pm 1$ for all $\chi \in \Irr(G)$ if and only if ${^\iota g}$ is conjugate to $g^{-1}$ for every $g \in G$.  It was proved by Kawanaka and Matsuyama \cite{KaMa90} that a formula for $\epi(\chi)$ is given by
$$ \epi(\chi) = \frac{1}{|G|} \sum_{g \in G} \chi(g \, {^\iota g}).$$
Further, by applying orthogonality relations, the following formula is obtained by Bump and Ginzburg \cite[Proposition 1]{BuGi04}:
$$ \sum_{\chi \in \Irr(G)} \epi(\chi) \chi(1) = \# \left\{ g \in G \, \mid \, {^\iota g} = g^{-1} \right\}.$$
An immediate consequence of this formula that we will need is as follows.

\begin{lemma} \label{Allepi1} If $G$ is a finite group with automorphism $\iota$ such that $\iota^2 = 1$, then $\epi(\chi) = 1$ for all $\chi \in \Irr(G)$ if and only if
$$ \sum_{\chi \in \Irr(G)} \chi(1) = \# \left\{ g \in G \, \mid \, {^\iota g} = g^{-1} \right\}.$$
\end{lemma}

When $\iota=1$, we recover the classical Frobenius-Schur indicator, $\epi(\chi) = \ep(\chi)$, in which case $\ep(\chi)=\pm 1$ if and only if $\chi$ is real-valued, and $\ep(\chi) = \pm 1$ for all $\chi \in \Irr(G)$ if and only if $G$ is a real group, so $g$ is conjugate to $g^{-1}$ for all $g \in G$.  Further, Lemma \ref{Allepi1} says that $\ep(\chi)=1$ for all $\chi \in \Irr(G)$, that is $G$ is \emph{totally orthogonal}, exactly when the sum of the degrees of the irreducible characters of $\chi$ is the number of involutions in $G$, where we define an involution of $G$ to be any element $g \in G$ such that $g^2 = 1$.  When $\iota$ is defined by an inner automorphism, say ${^\iota g} = h g h^{-1}$ with $h \in G$, and $h^2 = z \in Z(G)$ where $Z(G)$ is the center of $G$, we have an explicit relationship between $\epi(\chi)$ and $\ep(\chi)$.  In particular, if we let $\omega_{\chi}$ denote the central character of the representation of $G$ which affords $\pi$, we have \cite[Lemma 2.1]{Vi05}
\begin{equation} \label{central}
\epi(\chi) = \omega_{\chi}(z) \ep(\chi).
\end{equation}

Now suppose $H \leq G$ is an index 2 subgroup of $G$, such that for some $s \in G\setminus H$ with $s^2 = 1$, we have $G = H\langle s \rangle = H \cup sH$.  We have the following relationship between indicators and twisted indicators of $G$ and $H$ in this case.

\begin{lemma} \label{index2} Suppose $G$ is a finite group, $H \leq G$, $[G:H] = 2$, $G = H\langle s \rangle$ with $s^2 = 1$, and define $\iota$ on $H$ by ${^\iota h} = shs^{-1}$.  Then the following hold.
\begin{enumerate}
\item[(i)] If $\ep(\chi) = 1$ for all $\chi \in \Irr(G)$, and $H$ is a real group, then for all $\psi \in \Irr(H)$, we have $\ep(\psi) = 1$ and $\epi(\psi) \geq 0$.
\item[(ii)] If $\ep(\chi) = 1$ for all $\chi \in \Irr(G)$, and ${^\iota h}$ is $H$-conjugate to $h^{-1}$ for all $h \in H$, then for all $\psi \in \Irr(H)$, we have $\epi(\psi) = 1$ and $\ep(\psi) \geq 0$ .
\item[(iii)] If $G$ is a real group, and if $\ep(\psi) = 1$ for all $\psi \in \Irr(H)$,  then for all $\chi \in \Irr(G)$ we have $\ep(\chi) = 1$, and for all $\psi \in \Irr(H)$ we have $\epi(\psi) \geq 0$
\item[(iv)] If $G$ is a real group, and if $\epi(\psi) = 1$ for all $\psi \in \Irr(H)$, then for all $\chi \in \Irr(G)$ we have $\ep(\chi) = 1$, and for all $\psi \in \Irr(H)$ we have $\ep(\psi) \geq 0$.
\end{enumerate}
\end{lemma}
\begin{proof} All of the statements essentially follow from \cite[Proposition 2.1]{Vi10}, which may be applied as follows.  Given any $\psi \in \Irr(H)$, the induced character $\psi^G = \chi$ is either irreducible or else $\psi^G = \chi_1 + \chi_2$ where $\chi_1, \chi_2 \in \Irr(G)$.  Then \cite[Proposition 2.1]{Vi10} says that in these two cases, we have
\begin{equation} \label{epcond}
\ep(\chi) = \ep(\psi) + \epi(\psi), \quad \text{ or } \quad \ep(\chi_1) + \ep(\chi_2) = \ep(\psi) + \epi(\psi) \text{ with } \ep(\chi_1) = \ep(\chi_2),
\end{equation}
respectively.  

In cases (i) and (ii), we have $\ep(\chi) = \ep(\chi_1) = \ep(\chi_2) = 1$.  If $H$ is a real group, then $\ep(\psi) = \pm 1$ for all $\psi \in \Irr(G)$, and the only possibility for \eqref{epcond} to hold is that $\ep(\psi) =1$ and $\epi(\psi) \geq 0$ for all $\psi \in \Irr(G)$, and (i) follows.  Likewise for (ii), if ${^\iota h}$ is conjugate to $h^{-1}$ in $H$ for all $h \in H$, then $\epi(\psi) = \pm 1$ for all $\psi \in \Irr(G)$, and it follows from \eqref{epcond} we must have $\epi(\psi) = 1$ and $\ep(\psi) \geq 0$ for all $\psi \in \Irr(H)$.

In cases (iii) and (iv), we have $\ep(\chi) = \pm 1$ and $\ep(\chi_1)= \ep(\chi_2) = \pm 1$.  In (iii), since $\ep(\psi) = 1$ for all $\psi \in \Irr(H)$, we have \eqref{epcond} holds only if $\ep(\chi) =1$ for all $\chi \in \Irr(G)$ and $\epi(\psi) \geq 0$ for all $\psi \in \Irr(H)$.  Similarly in (iv), since $\epi(\psi) = 1$ for all $\psi \in \Irr(H)$, the only way \eqref{epcond} can be satisfied is if $\ep(\chi) = 1$ for all $\chi \in \Irr(G)$ (so $G$ is totally orthogonal) and $\ep(\psi) \geq 0$ for all $\psi \in \Irr(H)$.
\end{proof}

We will also need the following relationship between indicators of $G$ and $G/Z(G)$ when considering projective classical groups.

\begin{lemma} \label{epiGmodZ} Let $G$ be a finite group with automorphism $\iota$ such that $\iota^2 = 1$, and let $Z=Z(G)$ be the center of $G$.  Let $\iota$ also denote the automorphism of $G/Z$ defined by $\iota(gZ) = \iota(g) Z$.  If $\epi(\chi)=1$ for every $\chi \in \Irr(G)$, then $\epi(\psi) = 1$ for every $\psi \in \Irr(G/Z)$.
\end{lemma}
\begin{proof} Note that $\iota$ is well-defined on $G/Z$ by $\iota(gZ) = \iota(g) Z$.  Let $\psi \in \Irr(G/Z)$ and define the class function $\chi$ on $G$ by $\chi(g) = \psi(gZ)$.  It follows that $\chi \in \Irr(G)$, since
$$\langle \chi, \chi \rangle = \frac{1}{|G|} \sum_{g \in G} |\chi(g)|^2 = \frac{1}{|G|} \sum_{g \in G} |\psi(gZ)|^2 = \frac{|Z|}{|G|} \sum_{gZ \in G/Z} |\psi(gZ)|^2 = \langle \psi, \psi \rangle = 1.$$
By assumption we have $\epi(\chi) = 1$, and we can compute $\epi(\psi)$ as follows:
$$\epi(\psi) = \frac{|Z|}{|G|} \sum_{gZ \in G/Z} \psi( g \, {^\iota gZ}) = \frac{|Z|}{|G|} \sum_{gZ \in G/Z} \chi(g \, {^\iota g}) = \frac{1}{|G|} \sum_{g \in G} \chi(g \,  {^\iota g}) = \epi(\chi) = 1,$$
as claimed.
\end{proof}

\section{Real Simple Groups} \label{ReSimp}

Tiep and Zalesski classified all finite quasi-simple groups (and so all finite simple groups) which are real \cite{TiZa05}.  The following theorem, stated in \cite[Theorems 2 and 3]{VdGa10}, solves the problem of classifying all finite simple groups which are strongly real, and it turns out that a finite simple group is real if and only if it is strongly real.  This theorem follows from the work of many authors \cite{Ba87, Dj67, ElNo82, Ga10, Go81, KnTh98, KoNu05, Ra11, TiZa05, VdGa10, Wo66}.

\begin{theorem} \label{realsimple}  Let $G$ be a finite simple group.  Then $G$ is real if and only if $G$ is strongly real, and this occurs if and only if $G$ is isomorphic to one of the following groups:

(1) One of the Janko sporadic groups $J_1$ or $J_2$;

(2) One of the alternating groups $A_{10}$ or $A_{14}$;

(3) The Steinberg triality group ${^3 D}_4(q)$;

(4) $\Omega(2n+1,q)$ for $q \equiv 1($mod $4)$ and $n \geq 3$;

(5) $\Omega(9,q)$ for $q \equiv 3($mod $4)$;

(6) $\POm^-(4m,q)$ for $n \geq 2$;

(7) $\POm^+(4m, q)$ for $q \not\equiv 3($mod $4)$ and $n \geq 3$;

(8) $\POm^+(8,q)$;

(9) $\PSp(2n, q)$ for $q \not\equiv 3($mod $4)$ and $n \geq 1$.
\end{theorem}

As mentioned in the introduction, it has been conjectured that a finite simple group is strongly real if and only if it is totally orthogonal.  We now prove this statement for all but two types of groups, namely $\PSp(2n,q) = \Sp(2n,q)$ for $q$ even, and $\POm^{\pm}(4m,q) = \Om^{\pm}(4m, q)$ for $q$ even.  

\begin{theorem} \label{simple} Let $G$ be a finite simple group other than one of the form $\Sp(2n,q)$ with $q$ even or $\Omega^{\pm}(4m,q)$ with $q$ even.  Then $G$ is strongly real if and only if $G$ is totally orthogonal.  
\end{theorem}
\begin{proof}  First, if $G$ is totally orthogonal, then in particular $G$ is real, and from Theorem \ref{realsimple} we know $G$ is strongly real.

Conversely, suppose that $G$ is strongly real, so listed in Theorem \ref{realsimple}, but is not $\Sp(2n,q)$ with $q$ even or $\Omega^{\pm}(4m,q)$ with $q$ even.  We go through each type of group in the classification, and give references or show that it is totally orthogonal.

If $G$ is one of the Janko groups $J_1$ or $J_2$, then the fact that $G$ is totally orthogonal follows from the Atlas of finite groups \cite[pgs. 36, 42-43]{Atlas}.

If $G$ is $A_{10}$ or $A_{14}$, then $G$ is an index 2 subgroup of $S_{10}$ or $S_{14}$.  For all $n$, $S_n$ is totally orthogonal (all representations of $S_n$ are defined over $\mathbb{Q}$ by \cite[Theorem 2.1.12]{JaKe81}).  Since $G$ is real, it follows from Lemma \ref{index2}(i) that $G$ is totally orthogonal.

If $G={^3 D}_4(q)$, then for $q$ odd, $G$ is totally orthogonal by Barry \cite{Ba88}.  If $q$ is even, then $G$ is totally orthogonal by Ohmori \cite{Oh03}.

Consider the cases when either $G=\Omega(2n+1,q)$ with $n \geq 3, q \equiv 1($mod $4)$ or $G=\Omega(9,q)$ with $q \equiv 3($mod $4)$.  It is a result of Gow \cite[Theorem 2]{Go85} that $\SO(2n+1,q)$ is totally orthogonal when $q$ is odd.  Since $G$ is both a real group and an index 2 subgroup in some $\SO(2n+1,q)$, it follows from Lemma \ref{index2}(i) that $G$ is totally orthogonal.

Now consider the case when either $G=\POm^+(8, q)$ with $q$ odd, $G=\POm^-(4m,q)$ with $q$ odd and $m \geq 2$, or $G=\POm^+(4m,q)$ with $q \equiv 1($mod $4)$ and $m \geq 3$.    Gow \cite[Theorem 2]{Go85} proved that for $q$ odd and $m \geq 1$ the groups $\SO^{\pm}(4m, q)$ are all totally orthogonal, and so by Lemma \ref{epiGmodZ} the groups $\PSO^{\pm}(4m, q)$ are totally orthogonal.  Since in each case $G$ is an index 2 subgroup of some $\PSO^{\pm}(4m,q)$, and $G$ is a real group, then it follows from Lemma \ref{index2}(i) that $G$ is totally orthogonal.

Finally, we consider the case when $G = \PSp(2n,q)$ with $q \equiv 1($mod $4)$ and $n \geq 1$.  Gow proved \cite[Theorem 1]{Go85} that if $\chi$ is an irreducible character of $\Sp(2n,q)$ with $q \equiv 1($mod $4)$, then $\ep(\chi) = \omega_{\chi}(-I)$, where $\omega_{\chi}$ is the central character of the representation with character $\chi$.  Let $\alpha \in \FF_q$ such that $\alpha^2 = -1$.  If we define $\Sp(2n,q)$ via the alternating form corresponding to $\left[ \begin{array}{cc} 0 & -I_n \\ I_n & 0 \end{array} \right]$ (which may always be done), then $h = \left[\begin{array}{cc} \alpha I_n & 0 \\ 0 & -\alpha I_n \end{array} \right] \in \Sp(2n,q)$.  Then define $\iota$ to be the order 2 inner automorphism of $\Sp(2n,q)$ defined by conjugation by $h$.  Since $h^2 = -I$, then by \eqref{central} we have $\epi(\chi) = \omega_{\chi}(-I) \ep(\chi)$ for any irreducible character $\chi$ of $\Sp(2n,q)$.  From the result of Gow, for $q \equiv 1($mod $4)$ it follows that $\epi(\chi) = 1$ for every irreducible character $\chi$ of $\Sp(2n,q)$.  By Lemma \ref{epiGmodZ}, we have that $\epi(\psi)=1$ for any irreducible character $\psi$ of $G$, where $\iota$ is the induced inner automorphism defined by $hZ$ in $G$.  In this case, $(hZ)^2 = -IZ=Z$ in $G$, and so again by \eqref{central} we have $1=\epi(\psi) = \omega_{\psi}(Z) \ep(\psi) = \ep(\psi)$ for any irreducible character $\psi$ of $G$.  Thus $G$ is totally orthogonal.
\end{proof}

\section{Numbers of Involutions} \label{NumInvols}

In this section we count the number of involutions in various subgroups and cosets in the finite orthogonal groups.  In general, if $G$ is a group (or some subset of a group), then we will write $i(G)$ for the number of involutions in $G$.

\subsection{Groups over fields of odd characteristic}

We first consider the case that $q$ is odd.  Below we recall the number of involutions in the full orthogonal group $\gO^{\pm}(n,q)$, which is given in \cite[Lemma 6.1]{FuGuSt16}.  

\begin{proposition} \label{Oodd}
Let $q$ be odd.  
\begin{enumerate}
\item[(1)] The number of involutions in $\gO^{+}(2n,q)$ is equal to
$$ i(\gO^{+}(2n,q)) = \sum_{k=0}^{2n} \frac{ |\gO^{+}(2n,q)| }{ |\gO^+(k,q)| |\gO^+(2n-k,q)|} + \sum_{k=1}^{2n-1} \frac{ |\gO^{+}(2n,q)| }{ |\gO^-(k,q)| |\gO^-(2n-k,q)|}.$$
\item[(2)] The number of involutions in $\gO^{-}(2n,q)$ is equal to
$$ i(\gO^{-}(2n,q)) = \sum_{k=0}^{2n-1} \frac{ |\gO^-(2n,q)| }{ |\gO^+(k,q)| |\gO^-(2n-k,q)|} + \sum_{k=1}^{2n} \frac{ |\gO^{-}(2n,q)| }{ |\gO^-(k,q)| |\gO^+(2n-k,q)|}.$$
\item[(3)] The number of involutions in $\gO(2n+1,q)$ is equal to
$$i(\gO(2n+1,q)) = \sum_{k=0}^{2n+1} \frac{ |\gO(2n+1,q)| }{ |\gO^+(k,q)| |\gO^+(2n+1-k,q)|} + \sum_{k=1}^{2n} \frac{ |\gO(2n+1,q)| }{ |\gO^-(k,q)| |\gO^-(2n+1-k,q)|}.$$
\end{enumerate}
\end{proposition}

The quantities in Proposition \ref{Oodd} are obtained by noting that an involution in $\gO^{\pm}(N,q)$ has only $+1$ or $-1$ as eigenvalues.  If we fix the $-1$-eigenspace to have dimension $k$ and so the $+1$-eigenspace has dimension $N-k$, there are two conjugacy classes of involutions with this property, and their centralizers are products of smaller orthogonal groups whose orders are given in the denominators of the above expressions.  The type of orthogonal groups making up the direct product of their centralizer depends on the Witt type obtained when restricting the quadratic form to the eigenspaces.  Only slight modifications to this reasoning are needed to obtain expressions for other required involution counts.

\begin{proposition} \label{SOodd}
Let $q$ be odd.
\begin{enumerate}
\item[(1)] The number of involutions in $\SO^{+}(2n,q)$ is equal to
$$ i(\SO^{+}(2n,q)) = \sum_{R=0}^n \frac{ |\gO^{+}(2n,q)| }{ |\gO^+(2R,q)| |\gO^+(2n-2R,q)|} + \sum_{R=1}^{n-1} \frac{ |\gO^{+}(2n,q)| }{ |\gO^-(2R,q)| |\gO^-(2n-2R,q)|}.$$
\item[(2)] The number of involutions in $\SO^{-}(2n,q)$ is equal to
\begin{align*}
i(\SO^{-}(2n,q)) & = \sum_{R=0}^{n-1} \frac{ |\gO^{-}(2n,q)| }{ |\gO^+(2R,q)| |\gO^-(2n-2R,q)|} + \sum_{R=1}^{n} \frac{ |\gO^{-}(2n,q)| }{ |\gO^-(2R,q)| |\gO^+(2n-2R,q)|} \\
 & = 2 \sum_{R=0}^{n-1} \frac{ |\gO^{-}(2n,q)| }{ |\gO^+(2R,q)| |\gO^-(2n-2R,q)|}. 
\end{align*}
\item[(3)] The number of involutions in $\SO(2n+1,q)$ is equal to
\begin{align*}
i(\SO(2n+1,q)) & = \sum_{R=0}^{n} \frac{ |\gO(2n+1,q)| }{ |\gO^+(2R,q)| |\gO(2n+1-2R,q)|} + \sum_{R=1}^{n} \frac{ |\gO(2n+1,q)| }{ |\gO^-(2R,q)| |\gO(2n+1-2R,q)|} \\
& = \frac{1}{2} i(\gO(2n+1,q)).
\end{align*}
\end{enumerate}
\end{proposition}
\begin{proof}  In the expressions given in Proposition \ref{Oodd}, the index $k$ corresponds to the dimension of the $-1$-eigenspace of the involution.  The involution is in $\SO^{\pm}(N,q)$ if and only if $k$ is even, and so we let $k = 2R$ and take only these relevant terms from the sums in Proposition \ref{Oodd}.  In case (2), the two sums are the same after re-indexing.  In (3), the fact that $\gO(2n+1,q) \cong \SO(2n+1,q) \times \{ \pm 1 \}$ implies $i(\SO(2n+1,q)) = (1/2) i(\gO(2n+1,q))$.
\end{proof}

We also consider the number of involutions with determinant $-1$ in the orthogonal group of even dimension.

\begin{proposition} \label{SOcosetodd}
Let $q$ be odd.  The number of involutions in $\gO^{\pm}(2n,q)\setminus \SO^{\pm}(2n,q)$ is equal to 
 $$i(\gO^{\pm}(2n,q)\setminus \SO^{\pm}(2n,q))= 2\sum_{R=0}^{n-1} \frac{|\gO^{\pm}(2n,q)|}{|\gO(2R+1,q)||\gO(2n-2R-1,q)|}.$$
\end{proposition}
\begin{proof} This expression is obtained by considering when $k = 2R+1$ is odd in the sums in Proposition \ref{Oodd}.  Since $\gO^+(N,q) \cong \gO^-(N,q)$ when $N$ is odd, then the orders in the denominators of the two sums obtained are the same, and so the terms in each resulting sum are the same. 
\end{proof}

The following is needed to count involutions in the groups $\Omega^{\pm}(n,q)$ when $q$ is odd, and is a special case of \cite[Proposition 16.30]{CaEn04} (noting involutions are semisimple elements when $q$ is odd).

\begin{lemma} \label{OmCrit}
Let $q$ be odd, and let $C$ be a class of involutions in $\SO^{\pm}(N,q)$.  Let $d$ be the dimension of the $-1$-eigenspace of $C$ (where $d$ must be even).  If $Q$ is the quadratic form corresponding to $\SO^{\pm}(N,q)$, then let ${\bf s}_{-1}$ be the Witt type of the form $Q$ restricted to the $-1$-eigenspace of an involution in $C$. Define
\[v_- = \begin{cases}
\frac{d(q-1)}{4}    &    \text{ if } {\bf s}_{-1} = {\bf 0} \\
1 + \frac{d(q-1)}{4}    &    \text{ if } {\bf s}_{-1} = {\bf w}.
\end{cases}\]
Then $C$ lies in $\Omega^{\pm}(N,q)$ if and only if $v_-$ is an even integer, and $C$ defines a unique class of involutions in $\Omega^{\pm}(N,q)$ in this case.
\end{lemma}

We now have the following.

\begin{proposition} \label{OmOdd1mod4}
Let $q \equiv 1($mod $4)$.
\begin{enumerate}
\item[(1)] The number of involutions in $\Omega^+(2n,q)$ is equal to
  $$ i(\Om^+(2n,q)) = \sum_{R=0}^{n} \frac{|\gO^+(2n,q)|}{|\gO^+(2R,q)||\gO^+(2n-2R,q)|}.$$
\item[(2)] The number of involutions in $\Omega^-(2n,q)$ is equal to
  $$ i(\Om^-(2n,q)) =  \sum_{R=0}^{n-1} \frac{|\gO^-(2n,q)|}{|\gO^+(2R,q)||\gO^-(2n-2R,q)|}.$$
\item[(3)] The number of involutions in $\Omega(2n+1,q)$ is equal to
  $$ i(\Om(2n+1,q)) = \sum_{R=0}^{n} \frac{|\gO(2n+1,q)|}{|\gO^+(2R,q)||\gO(2n+1-2R,q)|}.$$
\end{enumerate}
\end{proposition}
\begin{proof} Applying Lemma \ref{OmCrit}, let $V_{-1}$ be the $-1$-eigenspace of some involution, where $\dim(V_{-1}) = d = 2R$.  Since $q \equiv 1($mod $4)$ and $d$ is even, then $v_-$ is always even if ${\bf s}_{-1} = {\bf 0}$ and $v_-$ is always odd if ${\bf s}_{-1}={\bf w}$.  Hence an involution lies in $\Omega(n,q)$ if and only if the quadratic form restricted to $V_{-1}$ has Witt type ${\bf 0}$, corresponding to the orthogonal group $\gO^+(2R,q)$. Taking the corresponding terms in the summations of Proposition \ref{Oodd}, we obtain the expressions as claimed.
\end{proof}

\begin{proposition} \label{OmOdd3mod4}
Let $q \equiv 3($mod $4)$. 
\begin{enumerate}
 \item[(1)] The number of involutions in $\Omega^+(2n,q)$ is equal to
  $$i( \Om^+(2n,q)) = \sum_{\substack{R=0 \\ R \text{ even}}}^{n} \frac{|\gO^+(2n,q)|}{|\gO^+(2R,q)||\gO^+(2n-2R,q)|} + \sum_{\substack{R=1 \\ R \text{ odd}}}^{n-1} \frac{|\gO^+(2n,q)|}{|\gO^-(2R,q)||\gO^-(2n-2R,q)|}.$$
  \item[(2)] The number of involutions in $\Omega^-(2n,q)$ is equal to
  $$i( \Om^-(2n,q)) = \sum_{\substack{R=0\\ R \text{ even}}}^{n-1} \frac{|\gO^-(2n,q)|}{|\gO^+(2R,q)||\gO^-(2n-2R,q)|} + \sum_{\substack{R=1 \\ R \text{ odd}}}^{n} \frac{|\gO^-(2n,q)|}{|\gO^-(2R,q)||\gO^+(2n-2R,q)|}.$$
  \item[(3)] The number of involutions in $\Omega(2n+1,q)$ is equal to
  $$ i(\Om(2n+1,q)) = \sum_{\substack{R=0 \\ R \text{ even}}}^{n} \frac{|\gO(2n+1,q)|}{|\gO^+(2R,q)||\gO(2n+1-2R,q)|} + \sum_{\substack{R=1 \\ R \text{ odd}}}^{n} \frac{|\gO(2n+1,q)|}{|\gO^-(2R,q)||\gO(2n+1-2R,q)|}.$$
  \end{enumerate}
\end{proposition}
\begin{proof} We use the same notation and idea as in the proof of Proposition \ref{OmOdd1mod4}, and apply Lemma \ref{OmCrit}.  Since $q \equiv 3($mod $4)$ and $d=2R$, then $v_-$ is even either when $R$ is even and ${\bf s}_{-1} = {\bf 0}$, or when $R$ is odd and ${\bf s}_{-1} = {\bf w}$.  These situations apply to the terms from Proposition \ref{Oodd} when either $R$ is even and the first factor in the denominator is $|\gO^+(2R,q)|$ (corresponding to ${\bf s}_{-1} = {\bf 0}$), or when $R$ is odd and the first factor in the denominator is $|\gO^-(2R,q)|$ (when ${\bf s}_{-1} ={\bf w}$), respectively.  The result follows.
\end{proof}

\subsection{Groups over fields of characteristic two}

We now deal with sets of involutions when $q$ is even.  The next result is \cite[Theorems 5.6]{FuGuSt16}.

\begin{proposition} \label{Oeven}
Let $q$ be even, and let
\begin{align*}
A_k^{\pm} & = q^{k(k-1)/2 + k(2n-2k)} |\Sp(k,q)|\,|\gO^{\pm}(2n-2k,q)|,\\
B_k & = 2q^{k(k+1)/2 + (k-1)(2n-2k)-1} q^{k-1}|\Sp(k-2,q)|\,|\Sp(2n-2k,q)|,\\
C_k & = 2q^{k(k-1)/2 + (k-1)(2n-2k)} |\Sp(k-1,q)|\,|\Sp(2n-2k,q)|.
\end{align*}
\begin{enumerate}
\item[(1)] The number of involutions in $\gO^+(2n,q)$ is equal to
$$i(\gO^+(2n,q))= \sum_{k=0 \atop{k \text{ even}}}^n \frac{|\gO^+(2n,q)|}{A^+_k} + \sum_{k=2 \atop{k \text{ even}}}^n \frac{|\gO^+(2n,q)|}{B_k} + \sum_{k=1 \atop{k \text{ odd}}}^n \frac{|\gO^+(2n,q)|}{C_k}.$$

\item[(2)] The number of involutions in $\gO^-(2n,q)$ is equal to
$$ i(\gO^-(2n,q)) = \sum_{k=0 \atop{k \text{ even}}}^{n-1} \frac{|\gO^-(2n,q)|}{A^-_k} + \sum_{k=2 \atop{k \text{ even}}}^n \frac{|\gO^-(2n,q)|}{B_k} + \sum_{k=1 \atop{k \text{ odd}}}^n \frac{|\gO^-(2n,q)|}{C_k},$$
\end{enumerate}
Note that in the $A_k$ sum in part (2), $r$ only ranges from $0$ to $n-1$, and that the values of $B_k$ and $C_k$ are the same in parts (1) and (2).
\end{proposition}

For the next result, we use Lemma \ref{OmCritqeven} and the ideas used to prove Proposition \ref{Oeven}.  We note that these results can also be concluded from \cite[Section 8]{Dy71}.

\begin{proposition} \label{Omeven}
Let $q$ be even, and let $A_k^{\pm}$, $B_k$, and $C_k$ be as in Proposition \ref{Oeven}.
\begin{enumerate}
\item[(1)] The number of involutions in $\Om^{+}(2n,q)$ is equal to
$$ i(\Om^+(2n,q)) =  \sum_{k=0 \atop{k \text{ even}}}^n \frac{|\gO^+(2n,q)|}{A_k^+} + \sum_{k=2 \atop{k \text{ even}}}^n \frac{|\gO^+(2n,q)|}{B_k}.$$
\item[(2)] The number of involutions in $\Om^{-}(2n,q)$ is equal to
$$ i(\Om^-(2n,q))= \sum_{k=0 \atop{k \text{ even}}}^{n-1} \frac{|\gO^-(2n,q)|}{A_k^-} + \sum_{k=2 \atop{k \text{ even}}}^n \frac{|\gO^-(2n,q)|}{B_k}. $$
\item[(3)] The number of involutions in $\gO^{\pm}(2n,q) \setminus \Om^{\pm}(2n,q)$ is
$$ i(\gO^{\pm}(2n,q) \setminus \Om^{\pm}(2n,q)) = \sum_{k=1 \atop{k \text{ odd}}}^n \frac{|\gO^{\pm}(2n,q)|}{C_k}.$$
\end{enumerate}
\end{proposition}
\begin{proof} In the proof of Proposition \ref{Oeven} given in \cite[Theorem 5.6]{FuGuSt16}, the authors prove that any involution in $\gO^{\pm}(2n,q)$ with $q$ even is conjugate to an element $g$, given in block form as
$$ g = \left[ \begin{array}{ccc} I & 0 & h \\ 0 & I & 0 \\ 0 & 0 & I \end{array} \right].$$
If we define $k = \mathrm{rank}(h)$, then the expressions in Proposition \ref{Oeven} count the involutions by considering the possible values for $k$.  By Lemma \ref{OmCritqeven}, the involution $g$ is in $\Om^{\pm}(2n,q)$ if and only if $\mathrm{rank}(1+g) = \mathrm{rank}(h)$ is even.  So, to count the involutions in $\Om^{\pm}(2n,q)$, we only include the sums with $k$ even from Proposition \ref{Oeven}, which gives (1) and (2).  The expression in (3) is the remaining sum.
\end{proof}

\section{Results on Indicators} \label{ResIndicators}

\subsection{$\SO^{\pm}(2n, q)$ with $q$ odd}

It is a result of Gow \cite{Go85} that for any $n \geq 1$ and $q$ odd, the group $\gO^{\pm}(n,q)$ is totally orthogonal, as is the group $\SO^{\pm}(4m,q)$ for $q$ odd and $m \geq 1$.  We expand these results to certain twisted indicators for the groups $\SO^{\pm}(4m+2, q)$ with $q$ odd.  We first need the following result, which essentially follows from the results of Wonenburger \cite{Wo66} along with the description of conjugacy classes in finite orthogonal groups due to Wall \cite{Wa62}.  A slightly stronger version of the following result is proved in \cite[Lemma 4.7]{ScVi16}, and statement (i) is proved by Gal't \cite[Theorem 1]{Ga10}, but the statements in the proposition below are also implied by the more general results of Kn\"{u}ppel  and Thomssen \cite{KnTh98}.

\begin{proposition} \label{detSO}
Let $q$ be odd, $n \geq 1$, $G = \gO^{\pm}(2n,q)$, and $H = \SO^{\pm}(2n, q)$.  
\begin{enumerate}
\item[(i)] Let $n = 2m$, $m \geq 1$, so $H = \SO^{\pm}(4m, q)$.  Then $H$ is a strongly real group.
\item[(ii)] Let $n = 2m+1$, $m \geq 0$, so $G = \gO^{\pm}(4m+2, q)$ and $H = \SO^{\pm}(4m+2, q)$.  For any $h \in H$, there exists an element $t \in G \setminus H$ such that $t h t^{-1} = h^{-1}$ and $t^2 = 1$.
\end{enumerate}
\end{proposition}

We apply Proposition \ref{detSO} to prove the following.  The result that $\SO^{\pm}(4m,q)$ is totally orthogonal, already proved by Gow, is included to show how it fits into the general framework.

\begin{theorem} \label{SOtwist} 
Let $q$ be odd, $n \geq 1$, $G = \gO^{\pm}(2n,q)$, and $H = \SO^{\pm}(2n, q)$.  Fix an involution $s \in G \setminus H$, and define $\iota$ on $H$ by ${^\iota h} = shs^{-1}$.
\begin{enumerate}
\item[(i)] Let $n = 2m$ so $H = \SO^{\pm}(4m, q)$.  Then $\ep(\psi)=1$ and $\epi(\psi) \geq 0$ for all $\psi \in \Irr(H)$.
\item[(ii)] Let $n = 2m+1$ so $H = \SO^{\pm}(4m+2,q)$.  Then $\epi(\psi) = 1$ and $\ep(\psi) \geq 0$ for all $\psi \in \Irr(H)$.  In particular, the character degree sum $\sum_{\psi \in \Irr(H)} \psi(1)$ is given by the expression in Proposition \ref{SOcosetodd}.
\end{enumerate}
\end{theorem}
\begin{proof} As already mentioned, if $q$ is odd, $n \geq 1$, and $G = \gO^{\pm}(2n,q)$, it is a result of Gow \cite[Theorem 1]{Go85} that $\ep(\chi) =1$ for all $\chi \in \Irr(G)$.  In case (i), since $H$ is a real group (in fact, strongly real) by Proposition \ref{detSO}(i), and $G = H\langle s \rangle$, then Lemma \ref{index2}(i) implies statement (i).

To prove statement (ii), it suffices to prove that for all $h \in H$, ${^\iota h}$ is $H$-conjugate to $h^{-1}$, by Lemma \ref{index2}(ii).  Given $h \in H$, we apply Proposition \ref{detSO} to ${^\iota h} \in H$, so there is some $t \in G \setminus H$ such that $t^2 = 1$ and $t({^\iota h})t = {^\iota h^{-1}} = s h^{-1} s$.  Conjugating by $s$ on both sides yields $(st) {^\iota h} (st)^{-1} = h^{-1}$, and since $s, t \in G \setminus H$, we have $st \in H$ and ${^\iota h}$ is $H$-conjugate to $h^{-1}$.  By Lemma \ref{Allepi1}, we have
$$ \sum_{\psi \in \Irr(H)} \psi(1) = \#\{ h \in H \, \mid \, s h s^{-1} = h^{-1} \}.$$
Note that since $s \in G \setminus H$, we have $sh \in G \setminus H$, and $shs^{-1} = h^{-1}$ if and only if $(sh)^2 = 1$.  That is, the character degree sum is equal to the number of involutions in $G \setminus H$, which is given by Proposition \ref{SOcosetodd}.
\end{proof}

\subsection{$\Sp(2n, q)$, $\gO^{\pm}(2n, q)$ and $\Om^{\pm}(2n, q)$ with $q$ even}

Using direct computation, we obtain the following result for the groups $\Sp(2n,q)$ with $n$ small.

\begin{theorem} \label{Spqeven}
Let $q$ be even and $n \leq 8$.  Then $G=\Sp(2n,q)$ is totally orthogonal.  In particular, Theorem \ref{simple} holds for these groups.
\end{theorem}
\begin{proof}  When $n=1$, this follows from \cite[Lemma 7]{Go76} together with \cite[Lemma 6.4]{Tu01}, for example.  When $n=2$, this follows from \cite[Lemma 7 and Theorem 8]{Go76}.  In general, from Lemma \ref{Allepi1}, the statement that $G$ is totally orthogonal is equivalent to 
$$ \sum_{\chi \in \Irr(G)} \chi(1) = \# \{ g \in G \, | \, g^2 = 1 \}.$$
The number of involutions in $\Sp(2n,q)$ with $q$ even is given in \cite[Theorem 5.3]{FuGuSt16}, and these expressions may be directly computed for $n \leq 8$.  Meanwhile, all character degrees of $\Sp(2n,q)$ with $n \leq 8$ have been computed by L\"{u}beck \cite{LuWWW}, and we may directly compute their sum to check that it is the same polynomial in $q$ as the number of involutions.  The results, computed with GAP \cite{GAP}, are given in the first table of the Appendix.
\end{proof}

Now consider the groups $\Omega^{\pm}(4m,q)$, which we suspect to be totally orthogonal.   It is further suspected that the groups $\gO^{\pm}(2n,q)$ with $q$ even are totally orthogonal, since they are known to be strongly real groups \cite{ElNo82, Go81}, and when $q$ is odd all of the groups $\gO^{\pm}(N, q)$ are known to be totally orthogonal \cite{Go85}.  We now relate the total orthogonality of the groups $\gO^{\pm}(2n,q)$ with $q$ even to indicators of the groups $\Om^{\pm}(2n,q)$.  Note the resemblance to the case when $q$ is odd given in Theorem \ref{SOtwist}. 

\begin{lemma} \label{OifOmega}
Let $q$ be even, $n \geq 1$, $G = \gO^{\pm}(2n, q)$, and $H = \Om^{\pm}(2n,q)$.  Fix an involution $s \in G \setminus H$, and define $\iota$ on $H$ by ${^\iota h} = shs^{-1}$.
\begin{enumerate}
\item[(i)] Let $n = 2m$ so $G = \gO^{\pm}(4m, q)$ and $H = \Om^{\pm}(4m, q)$.  If $\ep(\psi) = 1$ for all $\psi \in \Irr(H)$, then $\ep(\chi) = 1$ for all $\chi \in \Irr(G)$ and $\epi(\psi) \geq 0$ for all $\psi \in \Irr(H)$.  
\item[(ii)] Let $n = 2m+1$ so $G= \gO^{\pm}(4m+2, q)$ and $H = \Om^{\pm}(4m+2, q)$.  If $\epi(\psi) = 1$ for all $\psi \in \Irr(H)$, then $\ep(\chi) = 1$ for all $\chi \in \Irr(G)$ and $\ep(\psi) \geq 0$ for all $\psi \in \Irr(H)$.  
\end{enumerate}
\end{lemma}
\begin{proof}  Since $[G:H]=2$ and $G = H\langle s \rangle$, we may apply Lemma \ref{index2}.  It is proved by Gow \cite{Go81} and by Ellers and Nolte \cite{ElNo82} that $G = \gO^{\pm}(2n, q)$ is a strongly real group when $q$ is even, and in particular $G$ is a real group.  In the case that $n=2m$, statement (i) follows from Lemma \ref{index2}(iii), and in the case that $n=2m+1$, statement (ii) follows from Lemma \ref{index2}(iv).
\end{proof}

We again use computation to get the following results for $\Om^{\pm}(2n,q)$ and $\gO^{\pm}(2n,q)$ with $n$ small.

\begin{theorem} \label{OOmqeven}  Let $q$ be even.  Fix an involution $s \in \gO^{\pm}(2n,q) \setminus \Om^{\pm}(2n,q)$, and define $\iota$ on $\Om^{\pm}(2n,q)$ by ${^\iota h} = shs^{-1}$.
\begin{enumerate}
\item[(i)] If $n = 2m$, and $H=\Om^{\pm}(4m,q)$ with $m \leq3$, or $H=\Om^{+}(16,q)$, then $\ep(\psi)=1$ (and $\epi(\psi) \geq 0$) for all $\psi \in \Irr(H)$, so $H$ is totally orthogonal and Theorem \ref{simple} holds for these groups when they are simple.
\item[(ii)] If $n = 2m+1$, and $H = \Om^{\pm}(4m+2,q)$ with $m \leq 3$, then $\epi(\psi)=1$ (and $\ep(\psi) \geq 0$) for all $\psi \in \Irr(H)$.
\item[(iii)] If $G = \gO^{\pm}(2n,q)$ with $n \leq 7$, or $G = \gO^+(16,q)$, then $\ep(\chi)=1$ for all $\chi \in \Irr(G)$, so $G$ is totally orthogonal.
\end{enumerate}
\end{theorem}
\begin{proof}  In (i), by Lemma \ref{Allepi1} we have $\ep(\psi) = 1$ for all $\psi \in \Irr(H)$ if and only if the sum of the character degrees of $H$ is the number of involutions of $H=\Om^{\pm}(4m,q)$, which is counted in Proposition \ref{Omeven}(1) and (2).  In (ii), also by Lemma \ref{Allepi1}, we have $\epi(\psi)=1$ for all $\psi \in \Irr(H)$ if and only if the character degree sum of $H$ is the number of $h \in H$ such that $h \, {^\iota h} = 1$.  From the definition of $\iota$, we have $h \, {^\iota h} = (hs)^2$, where $hs \in \gO^{\pm}(4m+2, q) \setminus \Om^{\pm}(4m+2,q)$.  Thus the number of $h \in H$ such that $h \, {^\iota h} = 1$ is the number of involutions in $\gO^{\pm}(4m+2,q) \setminus \Om^{\pm}(4m+2, q)$, an expression for which is given in Proposition \ref{Omeven}(3).

We again apply the data of L\"{u}beck \cite{LuWWW}, who has computed all of the character degrees of $\Om^{\pm}(2n,q)$ when $q$ is even and $4 \leq n \leq 7$, and for $\Om^{+}(16, q)$ when $q$ is even, which are the groups $D_n(q)_{\SO}$ (for $\Omega^+(2n,q)$),  and ${^2 D}_n(q)_{\SO}$ (for $\Omega^-(2n,q)$) for $q \equiv 0$ mod $2$ in L\"{u}beck's data.  When $1 \leq n \leq 3$, we apply small rank group isomorphisms and known results.  

When $n=1$, we have by \cite[Proposition 2.9.1(iii)]{KlLi90} that $\gO^{\pm}(2,q)$ is isomorphic to the dihedral group of order $2(q \mp 1)$, and $\Omega^{\pm}(2,q)$ is isomorphic to the cyclic subgroup of rotations of order $q \mp 1$.  The number of involutions in $\gO^{\pm}(2,q) \setminus \Omega^{\pm}(2,q)$ is thus $q \mp 1$, corresponding to the reflections of the dihedral group, which is also the sum of the degrees of the cyclic group of order $q \mp 1$.  When $n = 2$, we have by \cite[Proposition 2.9.1(iv) and (v)]{KlLi90} that $\Omega^{+}(4,q) \cong \SL(2,q) \times \SL(2,q)$ (noting that $q$ is even) and $\Omega^{-}(4, q) \cong \SL(2,q^2)$.  As mentioned in the proof of Theorem \ref{Spqeven}, it is known that $\SL(2,q)$ is totally orthogonal when $q$ is even, and so the statement follows for $\Om^-(4,q)$.  The property of being totally orthogonal is preserved under taking direct products (see \cite[Lemma 2.2]{LuScVi12}, for example), and it follows that $\Omega^+(4,q)$ is also totally orthogonal.  When $n=3$, we have $\Omega^+(6,q) \cong \SL_4(q)$ and $\Omega^-(6,q) \cong \SU_4(q)$, from \cite[Proposition 2.9(vii)]{KlLi90}, noting that $Z(\SL_4(q)) = Z(\SU_4(q)) = 1$ when $q$ is even.  In this case, we may compute the number of involutions in $\gO^{\pm}(6,q) \setminus \Om^{\pm}(6,q)$ given by Proposition \ref{Omeven}(3) and match this to the character degree sums of $\SL_4(q)$ and $\SU_4(q)$, respectively, where these character degrees are given by L\"{u}beck's data for the groups $A_3(q)_{\mathrm{sc}}$ and ${^2 A}_3(q)_{\mathrm{sc}}$ for $q$ even.

For $4 \leq n \leq 8$, we compute the relevant number of involutions given in Proposition \ref{Omeven}, and match this with the character degree sums obtained from L\"{u}beck's data.  The polynomial in $q$ we obtain in each case for $1 \leq n \leq 8$ is given in the second table of the Appendix.  In particular we get the statements in (i) and (ii) that $\ep(\psi)=1$ or $\epi(\psi)=1$ for all $\psi \in \Irr(H)$.  The second claims in (i) and (ii) that $\epi(\psi) \geq 0$ or $\ep(\psi)\geq 0$ for all $\psi \in \Irr(H)$, and the statement in (iii) that the groups $\gO^{\pm}(2n,q)$ are totally orthogonal, both follow from Lemma \ref{OifOmega}.
\end{proof}

\section{Generating Functions} \label{GenFuns}

In this section we give generating functions for the numbers of involutions found in Section \ref{NumInvols}.  We use the following standard notation:
$$ (x; y)_n = \prod_{i = 1}^n (1 - xy^{i-1}), \quad (x; y)_{\infty} = \prod_{i \geq 1} (1 - xy^{i-1}) \;\; \text{ if } \;\; |y| < 1.$$
Then we can write
$$|\gO(2n+1, q)| = 2q^{2n^2+n} (1/q^2; 1/q^2)_n \; \text{ for } q \text{ odd,} \;  |\Sp(2n,q)| = q^{2n^2 + n} (1/q^2 ; 1/q^2)_n  \, \text{ for } q \text { even,}$$
$$\text{ and } |\gO^{\pm}(2n,q)| = \frac{2q^{2n^2}}{q^n \pm 1} (1/q^2 ; 1/q^2)_n \; \text{ for any } q.$$
From the above expression for $|\gO^-(2n,q)|$, for the purposes of convenience in power series coefficients, we will adopt the convention that $1/|\gO^-(0,q)| = 0$.

\subsection{Groups over fields of odd characteristic}

We begin with finding generating functions for the number of involutions in the groups $\SO^{\pm}(2n,q)$ for $q$ odd.  These essentially follow from computations made in \cite{FuGuSt16}.

\begin{theorem} \label{GFSOqodd}
For $q$ odd and $|u| < 1/q$, we have the generating functions
\begin{enumerate}
\item[(1)] {\hfil $\displaystyle \sum_{n \geq 0}  \frac{i(\SO^{\pm} (2n,q))}{ |\gO^{\pm}(2n,q)|}  q^{n^2} u^n= \frac{1}{2} \left[ \frac{ \prod_{i \geq 1} (1 + u/q^{2(i-1)})^2 }{\prod_{i \geq 1} (1 - u^2/q^{2(i-2)} ) } \pm \frac{ \prod_{i \geq 1} (1 + u/q^{2i-1})^2 }{ \prod_{i \geq 1} (1 - u^2/q^{2(i-1)}) } \right],$}
\item[(2)] {\hfil $\displaystyle \sum_{n \geq 0} \frac{i(\SO(2n+1,q))}{|\gO(2n+1,q)|} q^{n^2} u^n  = \frac{1}{2}\frac{1}{1-u} \frac{\prod_{i \geq 1} (1 + u/q^{2i})^2}{ \prod_{i \geq 1} (1 - u^2/q^{2i})}$.}
\end{enumerate}
\end{theorem}
\begin{proof} In \cite[CASE 1, Proof of Theorem 2.15]{FuGuSt16}, it is proved that
\begin{align*} 
S^+_1 & := \sum_{n \geq 0} u^n q^{n^2} \sum_{R=0}^n \frac{1}{|\gO^+(2R,q)||\gO^+(2n-2R,q)|} \\ 
           & = \frac{1}{4} \left[ \frac{(-u/q; 1/q^2)^2_{\infty}}{(u^2; 1/q^2)_{\infty}} + 2 \frac{(-u/q; 1/q^2)_{\infty} (-u; 1/q^2)_{\infty}}{(qu^2 ; 1/q^2)_{\infty}} + \frac{(-u; 1/q^2)^2_{\infty}}{(q^2u^2; 1/q^2)_{\infty}} \right],
\end{align*}
and in \cite[CASE 3, Proof of Theorem 2.15]{FuGuSt16} it is proved that
\begin{align*}
S^+_3 & := \sum_{n \geq 0} u^n q^{n^2} \sum_{R=1}^{n-1} \frac{1}{|\gO^-(2R,q)||\gO^-(2n-2R,q)|} \\ 
           & = \frac{1}{4} \left[ \frac{(-u/q; 1/q^2)^2_{\infty}}{(u^2; 1/q^2)_{\infty}} - 2 \frac{(-u/q; 1/q^2)_{\infty} (-u; 1/q^2)_{\infty}}{(qu^2 ; 1/q^2)_{\infty}} + \frac{(-u; 1/q^2)^2_{\infty}}{(q^2u^2; 1/q^2)_{\infty}} \right].
\end{align*}
By Proposition \ref{SOodd}(1) and the above, we have
\begin{align*}
\sum_{n \geq 0}  \frac{i(\SO^{+} (2n,q))}{ |\gO^{+}(2n,q)|}  q^{n^2} u^n & = S^+_1 + S^+_3  = \frac{1}{2}\left[ \frac{(-u; 1/q^2)^2_{\infty}}{(q^2u^2; 1/q^2)_{\infty}} + \frac{(-u/q; 1/q^2)^2_{\infty}}{(u^2; 1/q^2)_{\infty}}  \right] \\
 & = \frac{1}{2}\left[ \frac{ \prod_{i \geq 1} (1 + u/q^{2(i-1)})^2 }{\prod_{i \geq 1} (1 - u^2/q^{2(i-2)} ) } + \frac{ \prod_{i \geq 1} (1 + u/q^{2i-1})^2 }{ \prod_{i \geq 1} (1 - u^2/q^{2(i-1)}) } \right].
\end{align*}
In \cite[CASE 1, Proof of Theorem 2.16]{FuGuSt16}, it is proved that
$$S^-_1  := \sum_{n \geq 0} u^n q^{n^2} \sum_{R=0}^{n-1} \frac{1}{|\gO^+(2R,q)||\gO^-(2n-2R,q)|} = \frac{1}{4} \left[ \frac{(-u; 1/q^2)^2_{\infty}}{(q^2u^2; 1/q^2)_{\infty}} - \frac{(-u/q; 1/q^2)^2_{\infty}}{(u^2; 1/q^2)_{\infty}}  \right].$$
From this and Proposition \ref{SOodd}(2), we have
\begin{align*}
\sum_{n \geq 0} \frac{i(\SO^{-} (2n,q))}{ |\gO^{-}(2n,q)|} q^{n^2} u^n & = 2S^-_1  = \frac{1}{2}\left[ \frac{(-u; 1/q^2)^2_{\infty}}{(q^2u^2; 1/q^2)_{\infty}} - \frac{(-u/q; 1/q^2)^2_{\infty}}{(u^2; 1/q^2)_{\infty}}  \right] \\
 & = \frac{1}{2}\left[ \frac{ \prod_{i \geq 1} (1 + u/q^{2(i-1)})^2 }{\prod_{i \geq 1} (1 - u^2/q^{2(i-2)} ) } - \frac{ \prod_{i \geq 1} (1 + u/q^{2i-1})^2 }{ \prod_{i \geq 1} (1 - u^2/q^{2(i-1)}) } \right],
\end{align*}
which gives the result in (1).

For (2), it follows from \cite[Theorem 2.17]{FuGuSt16} and Proposition \ref{Oodd}(3) that
$$ \sum_{n \geq 0} \frac{i(\gO(2n+1,q))}{|\gO(2n+1,q)|} q^{n^2} u^n  = \frac{1}{1-u} \frac{\prod_{i \geq 1} (1 + u/q^{2i})^2}{ \prod_{i \geq 1} (1 - u^2/q^{2i})}.$$
Since $i(\SO(2n+1,q)) = (1/2) i(\gO(2n+1,q))$ from Proposition \ref{SOodd}(3), the generating function in (2) follows.
\end{proof}

From the result of Gow stated in Theorem \ref{SOtwist}(1) that $\SO^{\pm}(4m,q)$ is totally orthogonal when $q$ is odd, the character degree sum of $\SO^{\pm}(4m,q)$ is $|\gO^{\pm}(4m,q)|/q^{4m^2}$ times the coefficient of $u^{2m}$ of the generating function in Theorem \ref{GFSOqodd}(1).  By Theorem \ref{SOtwist}(2), the character degree sum of $\SO^{\pm}(4m+2,q)$ is $|\gO^{\pm}(4m+2,q)|/q^{(2m+1)^2}$ times the coefficient of $u^{2m+1}$ in the following generating function.

\begin{theorem} \label{GFO-SOqodd}
For $q$ odd and $|u| < 1/q$, we have the generating function
$$ \sum_{n \geq 0} \frac{i(\gO^{\pm} (2n,q) \setminus \SO^{\pm}(2n,q))}{ |\gO^{\pm}(2n,q)|} q^{n^2} = \frac{uq}{2} \frac{ \prod_{i \geq 1} (1 + u/q^{2(i-1)} )^2 } { \prod_{i \geq 1} (1 - u^2/q^{2(i-2)}) }.$$
\end{theorem}
\begin{proof} From \cite[Theorems 2.15 and 2.16]{FuGuSt16}, we have for $q$ odd and $|u|<1/q$,
$$ \sum_{n \geq 0} \frac{i(\gO^{\pm}(2n,q))}{ |\gO^{\pm}(2n,q)| }  q^{n^2} u^n = \frac{1}{2(1-uq)} \frac{ \prod_{i \geq 1} (1 + u/q^{2(i-1)})^2 }{\prod_{i \geq 1} (1 - u^2/q^{2(i-1)} ) } \pm \frac{1}{2}\frac{ \prod_{i \geq 1} (1 + u/q^{2i-1})^2 }{ \prod_{i \geq 1} (1 - u^2/q^{2(i-1)}) }.$$
The result is obtained by taking the difference of this and the generating function in Theorem \ref{GFSOqodd}(1).
\end{proof}

Generating functions for the number of involutions in $\Om^{\pm}(n,q)$ with $q \equiv 1($mod $4)$ also follow quickly from calculations in \cite{FuGuSt16}, because of the convenient formula for the number of involutions in this case, given in Proposition \ref{OmOdd1mod4}.

\begin{theorem} \label{GFOmq1mod4}
For $q \equiv 1($mod $4)$ and $|u| < 1/q$, we have the generating functions
\begin{enumerate}
\item[(1)]
{\hfil $\displaystyle \sum_{n \geq 0} \frac{i(\Om^+(2n,q))}{|\gO^+(2n,q)|}q^{n^2} u^n$}  
$$=  \frac{1}{4} \left[ \frac{\prod_{i \geq 1} (1 + u/q^{2(i-1)})^2}{ \prod_{i \geq 1} (1 - u^2/q^{2(i-2)})} + \frac{2 \prod_{i \geq 1} (1 + u/q^{i-1})}{ \prod_{i \geq 1} (1 - u^2/q^{2i-3}) } + \frac{\prod_{i \geq 1} (1 + u/q^{2i-1})^2}{ \prod_{i \geq 1} (1 - u^2/q^{2(i-1)})} \right],$$
\item[(2)]
{\hfil $\displaystyle \sum_{n \geq 0} \frac{i(\Om^-(2n,q))}{|\gO^-(2n,q)|} q^{n^2} u^n  = \frac{1}{4} \left[  \frac{\prod_{i \geq 1} (1 + u/q^{2(i-1)})^2}{ \prod_{i \geq 1} (1 - u^2/q^{2(i-2)})} - \frac{\prod_{i \geq 1} (1 + u/q^{2i-1})^2}{ \prod_{i \geq 1} (1 - u^2/q^{2(i-1)})}  \right],$}
\item[(3)]
{\hfil $\displaystyle \sum_{n \geq 0} \frac{i(\Om(2n+1,q))} { |\gO(2n+1,q)| } q^{n^2} u^n   =  \frac{1}{4} \left[ \frac{(1 +u) \prod_{i \geq 1} (1 + u/q^{2i})^2}{ \prod_{i \geq 1} (1 - u^2/q^{2(i-1)}) } + \frac{\prod_{i \geq 1} (1 + u/q^{i})}{ \prod_{i \geq 1} (1 - u^2/q^{2i-1}) } \right].$}
\end{enumerate}
\end{theorem}
\begin{proof} For (1), as in the proof of Theorem \ref{GFSOqodd}, it was proved in \cite{FuGuSt16} that
$$ \sum_{n \geq 0} u^n q^{n^2} \sum_{R=0}^n \frac{1}{|\gO^+(2R,q)||\gO^+(2n-2R,q)|} $$
$$          = \frac{1}{4} \left[ \frac{(-u/q; 1/q^2)^2_{\infty}}{(u^2; 1/q^2)_{\infty}} + 2 \frac{(-u/q; 1/q^2)_{\infty} (-u; 1/q^2)_{\infty}}{(qu^2 ; 1/q^2)_{\infty}} + \frac{(-u; 1/q^2)^2_{\infty}}{(q^2u^2; 1/q^2)_{\infty}} \right].$$
The result follows by simplifying this expression, noting that
$$ (-u/q; 1/q^2)_{\infty} (-u; 1/q^2)_{\infty} = \prod_{i \geq 1} (1 + u/q^{2i-1}) \prod_{i \geq 1} (1 + u/q^{2i-2}) = \prod_{i \geq 1} (1 + u/q^{i-1}),$$
together with Proposition \ref{OmOdd1mod4}(1).

The generating function in (2) follows directly from Proposition \ref{OmOdd1mod4}(2) and Theorem \ref{GFSOqodd}(1), since $i(\Om^-(2n,q)) = (1/2) i(\SO^-(2n,q))$ when $q \equiv 1($mod $4)$.

For (3), we use the identity proved in \cite[CASE 1, Proof of Theorem 2.17]{FuGuSt16} that
\begin{align*}
\sum_{n \geq 0} u^n q^{n^2} & \left[ \sum_{R=0}^n \frac{1}{ |\gO^+(2R,q)| |\gO(2n+1-2R,q)|} \right]  \\
 & = \frac{1}{4} \left[\frac{(-u/q^2; 1/q^2)_{\infty} (-u ; 1/q^2)_{\infty}}{(u^2 ; 1/q^2)_{\infty}}  + \frac{(-u/q^2; 1/q^2)_{\infty} (-u/q ; 1/q^2)_{\infty}}{(u^2/q ; 1/q^2)_{\infty}}  \right] \\
&= \frac{1}{4} \left[ \frac{\prod_{i \geq 1} (1 + u/q^{2i}) \prod_{i \geq 1} (1 + u/q^{2(i-1)})}{\prod_{i \geq 1} (1 - u^2/q^{2(i-1)})} + \frac{\prod_{i \geq 1} (1 + u/q^{2i}) \prod_{i \geq 1} (1 + u/q^{2i-1})} {\prod_{i \geq 1} (1 - u^2/q^{2i-1}) } \right] \\
& = \frac{1}{4} \left[ \frac{(1 +u) \prod_{i \geq 1} (1 + u/q^{2i})^2}{ \prod_{i \geq 1} (1 - u^2/q^{2(i-1)}) } + \frac{\prod_{i \geq 1} (1 + u/q^{i})}{ \prod_{i \geq 1} (1 - u^2/q^{2i-1}) } \right].
\end{align*}
The result follows by Proposition \ref{OmOdd1mod4}(3).
\end{proof}

The last case that we deal with in this section, which is for $\Om^{\pm}(n,q)$ with $q \equiv 3($mod $4)$, does not immediately follow from calculations made in \cite{FuGuSt16}.  We apply the $q$-binomial theorem and a slight variant, in the following form.  A proof of the first statement is in \cite[pg. 17]{An}, and the second statement is \cite[Corollary 2.3]{FuGuSt16}.

\begin{lemma} \label{qbinom} If $|y| < 1$, then
\begin{enumerate}
\item[(1)]  ${\hfil \displaystyle \sum_{n \geq 0} \frac{(A; y)_n}{(y; y)_n} x^n = \frac{(Ax; y)_{\infty}}{(x; y)_{\infty}}, }$ if $|x| < 1$, and
\item[(2)] ${\hfil \displaystyle \sum_{n \geq 0} \frac{y^{\binom{n}{2}}}{(y; y)_n} x^n = \frac{1}{(-x; y)_{\infty}}. }$
\end{enumerate}
\end{lemma}

We may now prove the following.

\begin{theorem} \label{GFOmq3mod4}
Let $q \equiv 3($mod $4)$.
\begin{enumerate}
\item[(1)] For $|u| < 1/q$, we have the generating function
\begin{align*}
\sum_{n \geq 0} \frac{i(\Om^{\pm}(2n,q))} {|\gO^{\pm}(2n,q)|} q^{n^2} u^n& = \frac{1}{4} \left[ \frac{\prod_{i \geq 1} (1 + u/q^{2(i-1)})^2}{ \prod_{i \geq 1} (1 - u^2/q^{2(i-2)})}  + \frac{ \prod_{i \geq 1} (1 + (-1)^{i-1}u/q^{i-1})}{ \prod_{i \geq 1} (1 + u^2/q^{2i-3})} \right.\\
&\left. \pm \frac{ \prod_{i \geq 1} (1 - (-1)^{i-1}u/q^{i-1})}{ \prod_{i \geq 1} (1 + u^2/q^{2i-3})} \pm \frac{ \prod_{i \geq 1} (1 + u/q^{2i-1})^2}{ \prod_{i \geq 1} (1 - u^2/q^{2(i-1)})}\right].
\end{align*}
\item[(2)] For $|u| < 1$, we have the generating function
$$ \sum_{n \geq 0} \frac{i(\Om(2n+1,q))}{|\gO(2n+1,q)|} q^{n^2} u^n = \frac{1}{4} \left[ \frac{(1+u)\prod_{i \geq 1} (1 +  u/q^{2i})^2} { \prod_{i \geq 1} (1 - u^2/q^{2(i-1)})} + \frac{\prod_{i \geq 1}(1 + (-1)^i u/q^{i}) } { \prod_{i \geq 1} (1 + u^2/q^{2i-1})} \right].$$
\end{enumerate}
\end{theorem}
\begin{proof}
For (1), by Proposition \ref{OmOdd3mod4}(1) and (2), we must find generating functions for
$$ \sum_{n \geq 0} u^n q^{n^2} \left[ \sum_{R=0 \atop{R \text{ even}}}^{n} \frac{1}{|\gO^+(2R,q)||\gO^+(2n-2R,q)|} + \sum_{R=1 \atop{R \text{ odd}}}^{n-1} \frac{1}{|\gO^-(2R,q)||\gO^-(2n-2R,q)|} \right]$$
and
$$ \sum_{n \geq 0} u^n q^{n^2} \left[ \sum_{R=0 \atop{R \text{ even}}}^{n-1} \frac{1}{|\gO^+(2R,q)||\gO^-(2n-2R,q)|} + \sum_{R=1 \atop{R \text{ odd}}}^{n} \frac{1}{|\gO^-(2R,q)||\gO^+(2n-2R,q)|} \right].$$
Consider first the sum
\begin{align*}
S_1  :&= \sum_{n \geq 0} u^n q^{n^2} \sum_{R=0 \atop{R \text{ even}}}^n \frac{1}{|\gO^+(2R,q)||\gO^+(2n-2R,q)|} = \sum_{n \geq 0} u^n q^{n^2} \sum_{R=0}^{\lfloor n/2 \rfloor} \frac{1}{|\gO^+(4R,q)||\gO^+(2n-4R,q)|} \\
& = \sum_{n \geq 0} u^n q^{n^2} \sum_{R=0}^{\lfloor n/2 \rfloor} \frac{(q^{2R}+1)(q^{n-2R} + 1)}{ 4 q^{8R^2} (1/q^2 ; 1/q^2)_{2R} \, q^{2(n-2R)^2} (1/q^2 ; 1/q^2)_{n-2R}} \\
& = \sum_{R \geq 0} \sum_{n \geq 0} u^n q^{n^2} \frac{(q^{2R}+1)(q^{n-2R} + 1)}{ 4 q^{8R^2} (1/q^2 ; 1/q^2)_{2R} \, q^{2(n-2R)^2} (1/q^2 ; 1/q^2)_{n-2R}}.
\end{align*}
Now replace $n$ by $n+2R$ to obtain
\begin{align*}
S_1 & = \sum_{R \geq 0} \sum_{n \geq 0} u^{n+2R} q^{(n+2R)^2} \frac{(q^{2R} + 1)(q^n + 1)}{4 q^{8R^2} (1/q^2 ; 1/q^2)_{2R} \, q^{2n^2} (1/q^2 ; 1/q^2)_n} \\
& = \sum_{R \geq 0} \sum_{n \geq 0} u^{n+2R} \frac{q^{4nR}(q^{n + 2R} + q^n + q^{2R} + 1)}{4 q^{4R^2} (1/q^2 ; 1/q^2)_{2R} \, q^{n^2} (1/q^2 ; 1/q^2)_n} \\
& = \sum_{R \geq 0} \frac{u^{2R}}{4q^{4R^2} (1/q^2 ; 1/q^2)_{2R}} \sum_{n \geq 0} \frac{q^{2R} (uq^{4R})^n + (uq^{4R})^n + q^{2R} (uq^{4R-1})^n + (uq^{4R-1})^n}{ q^{2 \binom{n}{2}} (1/q^2 ; 1/q^2)_n}.
\end{align*}
Now apply Lemma \ref{qbinom}(2) with $y=1/q^2$ to each of the terms in the summation over $n$.  This gives
$$ S_1 = \sum_{R \geq 0} u^{2R} \frac{(-uq^{4R}; 1/q^2)_{\infty} (q^{2R} + 1) + (-uq^{4R-1}; 1/q^2)_{\infty} (q^{2R} + 1)}{4 q^{4R^2} (1/q^2 ; 1/q^2)_{2R}}.$$
We apply the facts that
$$ (-uq^{4R}; 1/q^2)_{\infty} = u^{2R} q^{4R^2 + 2R} (-1/uq^2; 1/q^2)_{2R} (-u; 1/q^2)_{\infty}, \, \text{ and }$$
$$ (-uq^{4R-1} ; 1/q^2)_{\infty} = u^{2R} q^{4R^2} (-1/uq; 1/q^2)_{2R} (-u/q; 1/q^2)_{\infty},$$
to rewrite $S_1$ as
\begin{align}
S_1 = & \frac{(-u; 1/q^2)_{\infty} }{4} \sum_{R \geq 0} \frac{(u^2 q^2)^{2R} (-1/uq^2 ; 1/q^2)_{2R} + (u^2 q)^{2R} (-1/uq^2 ; 1/q^2)_{2R}}{(1/q^2 ; 1/q^2)_{2R}}  \nonumber\\
& + \frac{(-u/q; 1/q^2)_{\infty} }{4} \sum_{R \geq 0} \frac{(u^2 q)^{2R} (-1/uq ; 1/q^2)_{2R} + (u^2)^{2R} (-1/uq ; 1/q^2)_{2R}}{(1/q^2 ; 1/q^2)_{2R}}. \label{S1}
\end{align}

Now we consider the sum
\begin{align*}
T_1 :&= \sum_{n \geq 0} u^n q^{n^2} \sum_{R=1 \atop{R \text{ odd}}}^{n-1} \frac{1}{|\gO^-(2R,q)||\gO^-(2n-2R,q)|}  \\
& = \sum_{n \geq 1} u^n q^{n^2} \sum_{R=0}^{\lfloor (n-1)/2 \rfloor} \frac{1}{|\gO^-(4R+2,q)||\gO^-(2n - 4R-2,q)|} \\
& = \sum_{R \geq 0} \sum_{n \geq 1} u^n q^{n^2} \frac{(q^{2R+1} - 1)(q^{n-2R-1} - 1)}{ 4 q^{2 (2R+1)^2} (1/q^2 ; 1/q^2)_{n-2R-1} \, q^{2(n-2R-1)^2} (1/q^2; 1/q^2)_{n-2R-1}}.
\end{align*}
Replace $n$ with $n+2R+1$ to rewrite $T_1$ as
\begin{align*}
T_1 & = \sum_{R \geq 0} \sum_{n \geq 0} u^{n+2R+1} q^{(n+2R+1)^2} \frac{(q^{2R+1} - 1)(q^n - 1)}{4 q^{2(2R+1)^2} (1/q^2 ; 1/q^2)_{2R+1} q^{2n^2} (1/q^2 ; 1/q^2)_n} \\
& = \sum_{R \geq 0} \frac{u^{2R+1}}{4 q^{(2R+1)^2} (1/q^2 ; 1/q^2)_{2R+1}} \sum_{n \geq 0} \frac{q^{2R+1} (uq^{4R+2})^n - (uq^{4R+2})^n - q^{2R+1} (uq^{4R+1})^n + (uq^{4R+1})^n}{q^{2 \binom{n}{2}} (1/q^2 ; 1/q^2)_n}.
\end{align*}
As we did with $S_1$, we apply Lemma \ref{qbinom}(2) with $y=1/q^2$ to each term in the summation over $n$, and we obtain
$$ T_1 = \sum_{R \geq 0} u^{2R+1}\frac{ (-uq^{4R+2}; 1/q^2)_{\infty} (q^{2R+1} - 1) - (-uq^{4R+1}; 1/q^2)_{\infty} (q^{2R+1} - 1)}{4q^{(2R+1)^2} (1/q^2 ; 1/q^2)_{2R+1}}.$$
Now apply the identities
$$ (-uq^{4R+2}; 1/q^2)_{\infty} = u^{2R+1} q^{(2R+1)(2R+2)} (-1/uq^2; 1/q^2)_{2R+1} (-u; 1/q^2)_{\infty}, \, \text{ and }$$
$$ (-uq^{4R+1}; 1/q^2)_{\infty} = u^{2R+1} q^{(2R+1)^2} (-1/uq; 1/q^2)_{2R+1} (-u/q; 1/q^2)_{\infty}$$
to rewrite $T_1$ as
\begin{align}
T_1 = & \frac{(-u; 1/q^2)_{\infty}}{4} \sum_{R \geq 0} \frac{ (u^2 q^2)^{2R+1} (-1/uq^2; 1/q^2)_{2R+1} + (-u^2 q)^{2R+1} ( -1/uq^2; 1/q^2)_{2R+1}}{(1/q^2; 1/q^2)_{2R+1}} \nonumber \\
& + \frac{(-u/q; 1/q^2)_{\infty}}{4} \sum_{R \geq 0} \frac{ (-u^2 q)^{2R+1} (-1/uq; 1/q^2)_{2R+1} + (u^2)^{2R+1} ( -1/uq; 1/q^2)_{2R+1}}{(1/q^2; 1/q^2)_{2R+1}}. \label{T1}
\end{align}
By combining \eqref{S1} and \eqref{T1} as even and odd terms, we have
\begin{align*}
S_1 + T_1 = & \frac{(-u; 1/q^2)_{\infty}}{4} \sum_{R \geq 0} \frac{(-1/uq^2; 1/q^2)_{R}}{(1/q^2; 1/q^2)_R} \left[ (u^2 q^2)^R + (-u^2 q)^R \right] \\
& + \frac{(-u/q; 1/q^2)_{\infty}}{4} \sum_{R \geq 0} \frac{(-1/uq; 1/q^2)_{R}}{(1/q^2; 1/q^2)_R} \left[ (-u^2 q)^R + (u^2 )^R \right].
\end{align*}
Now apply Lemma \ref{qbinom}(1) to each term in the summations over $R$, with $y = 1/q^2$, $A = -1/uq^2, -1/uq$, and $x = u^2q^2, -u^2 q, -u^2$, to obtain
\begin{align}
S&_1  + T_1 =  \label{S1T1} \\
& = \frac{1}{4} \left[ \frac{(-u; 1/q^2)^2_{\infty} }{(u^2 q^2 ; 1/q^2)_{\infty}} + \frac{(-u; 1/q^2)_{\infty} (u/q; 1/q^2)_{\infty} + (u; 1/q^2)_{\infty} (-u/q ; 1/q^2)_{\infty}}{(-u^2q; 1/q^2)_{\infty}} + \frac{(-u/q; 1/q^2)^2_{\infty}}{(u^2 ; 1/q^2)_{\infty}} \right], \nonumber
\end{align}
where $|u|< 1/q$ is the smallest radius of convergence coming from Lemma \ref{qbinom}(1).

Next consider 
\begin{align*}
S_2 :&= \sum_{n \geq 0} u^n q^{n^2} \sum_{R=0 \atop{R \text{ even}}}^{n-1} \frac{1}{|\gO^+(2R,q)||\gO^-(2n-2R,q)|} = \sum_{n \geq 0} \sum_{R=0}^{\lfloor (n-1)/2 \rfloor} \frac{u^n q^{n^2}} {|\gO^+(4R,q)||\gO^-(2n-4R,q)|} \\
&  = \sum_{R \geq 0} \sum_{n \geq 0} u^n q^{n^2} \frac{(q^{2R}+1)(q^{n-2R} - 1)}{ 4 q^{8R^2} (1/q^2 ; 1/q^2)_{2R} \, q^{2(n-2R)^2} (1/q^2 ; 1/q^2)_{n-2R}}.
\end{align*}
We note this is the same expression we obtained for $S_1$, except for the $(q^{n-2R} - 1)$ factor instead of $(q^{n-2R} + 1)$ in the numerator.  Repeating the same calculations made for $S_1$, while keeping track of the sign change, gives
\begin{align}
S_2 = & \frac{(-u; 1/q^2)_{\infty} }{4} \sum_{R \geq 0} \frac{(u^2 q^2)^{2R} (-1/uq^2 ; 1/q^2)_{2R} + (u^2 q)^{2R} (-1/uq^2 ; 1/q^2)_{2R}}{(1/q^2 ; 1/q^2)_{2R}}  \nonumber\\
& - \frac{(-u/q; 1/q^2)_{\infty} }{4} \sum_{R \geq 0} \frac{(u^2 q)^{2R} (-1/uq ; 1/q^2)_{2R} + (u^2)^{2R} (-1/uq ; 1/q^2)_{2R}}{(1/q^2 ; 1/q^2)_{2R}}. \label{S2}
\end{align}
Similarly, we consider
\begin{align*}
T_2 :&= \sum_{n \geq 0} u^n q^{n^2} \sum_{R=1 \atop{R \text{ odd}}}^{n} \frac{1}{|\gO^-(2R,q)||\gO^+(2n-2R,q)|}  = \sum_{n \geq 1} \sum_{R \geq 0}  \frac{u^n q^{n^2}}{|\gO^-(4R+2, q)| |\gO^+(2n-4R-2,q)|}\\
& = \sum_{R \geq 0} \sum_{n \geq 1} u^n q^{n^2} \frac{(q^{2R+1} - 1)(q^{n-2R-1} + 1)}{ 4 q^{2 (2R+1)^2} (1/q^2 ; 1/q^2)_{n-2R-1} \, q^{2(n-2R-1)^2} (1/q^2; 1/q^2)_{n-2R-1}}.
\end{align*}
We again have an expression that is almost exactly the same as $T_1$, except for the $(q^{n-2R-1} + 1)$ factor in the numerator instead of $(q^{n-2R-1} - 1)$.  Following the calculation for $T_1$ and making the appropriate sign changes yields
\begin{align}
T_2 = & \frac{(-u; 1/q^2)_{\infty} }{4} \sum_{R \geq 0} \frac{(u^2 q^2)^{2R+1} (-1/uq^2 ; 1/q^2)_{2R+1} + (-u^2 q)^{2R+1} (-1/uq^2 ; 1/q^2)_{2R+1}}{(1/q^2 ; 1/q^2)_{2R+1}}  \nonumber\\
& - \frac{(-u/q; 1/q^2)_{\infty} }{4} \sum_{R \geq 0} \frac{(-u^2 q)^{2R+1} (-1/uq ; 1/q^2)_{2R+1} + (u^2)^{2R+1} (-1/uq ; 1/q^2)_{2R+1}}{(1/q^2 ; 1/q^2)_{2R+1}}. \label{T2}
\end{align}
From \eqref{S2} and \eqref{T2} we have
\begin{align*}
S_2 + T_2 &  =\frac{(-u; 1/q^2)_{\infty}}{4} \sum_{R \geq 0} \frac{(-1/uq^2; 1/q^2)_{R}}{(1/q^2; 1/q^2)_R} \left[ (u^2 q^2)^R + (-u^2 q)^R \right] \\
& - \frac{(-u/q; 1/q^2)_{\infty}}{4} \sum_{R \geq 0} \frac{(-1/uq; 1/q^2)_{R}}{(1/q^2; 1/q^2)_R} \left[ (-u^2 q)^R + (u^2 )^R \right].
\end{align*}
Apply Lemma \ref{qbinom}(1) for each term, and we have
\begin{align}
S&_2  + T_2 = \label{S2T2} \\
& = \frac{1}{4} \left[ \frac{(-u; 1/q^2)^2_{\infty} }{(u^2 q^2 ; 1/q^2)_{\infty}} + \frac{(-u; 1/q^2)_{\infty} (u/q; 1/q^2)_{\infty} - (u; 1/q^2)_{\infty} (-u/q ; 1/q^2)_{\infty}}{(-u^2q; 1/q^2)_{\infty}} - \frac{(-u/q; 1/q^2)^2_{\infty}}{(u^2 ; 1/q^2)_{\infty}} \right], \nonumber
\end{align}
when $|u|< 1/q$.

Finally, the expressions \eqref{S1T1} and \eqref{S2T2} simplify to the generating function claimed in (1), after noticing that
$$ (\mp u; 1/q^2)_{\infty} ( \pm u/q; 1/q^2)_{\infty} = \prod_{i \geq 1} (1 \pm u/q^{2(i-1)}) (1 \mp u/q^{2i-1}) = \prod_{i \geq 1} (1 \pm (-1)^{i-1} u/q^{i-1}).$$

For (2), by Proposition \ref{OmOdd3mod4}(3) we need the generating function for
$$ \sum_{n \geq 0} u^n q^{n^2} \left[ \sum_{R=0 \atop{R \text{ even}}}^{n} \frac{1}{|\gO^+(2R,q)||\gO(2n+1-2R,q)|} + \sum_{R=1 \atop{R \text{ odd}}}^{n} \frac{1}{|\gO^-(2R,q)||\gO(2n+1-2R,q)|} \right].$$
We proceed as we did in (1), and we first consider
\begin{align*}
S_3 :& = \sum_{n \geq 0} u^n q^{n^2} \sum_{R=0 \atop{R \text{ even}}}^{n} \frac{1}{|\gO^+(2R,q)||\gO(2n+1-2R,q)|} = \sum_{n \geq 0} \sum_{R=0}^{\lfloor n/2 \rfloor} \frac{u^n q^{n^2}}{ |\gO^+(4R,q)| |\gO(2n+1 - 4R, q)| } \\
& = \sum_{R \geq 0} \sum_{n \geq 0} \frac{u^n q^{n^2} (q^{2R} + 1)}{ 4q^{8R^2} (1/q^2 ; 1/q^2)_{2R} \, q^{2(n-2R)^2 + (n-2R)} (1/q^2 ; 1/q^2)_{n-2R}}.
\end{align*}
We replace $n$ with $n+2R$ and simplify to obtain
\begin{align*}
S_3 & = \sum_{R \geq 0} \sum_{n \geq 0} \frac{u^{n+2R} q^{(n+2R)^2} (q^{2R} + 1)}{4q^{8R^2} (1/q^2; 1/q^2)_{2R} q^{2n^2  + n} (1/q^2 ; 1/q^2)_n} \\
& = \sum_{R \geq 0} \frac{u^{2R}}{4q^{4R^2} (1/q^2 ; 1/q^2)_{2R}} \sum_{n \geq 0} \frac{q^{2R}(uq^{4R-2})^n + (uq^{4R-2})^n}{ q^{2 \binom{n}{2}} (1/q^2 ; 1/q^2)_n}.
\end{align*}
By applying Lemma \ref{qbinom}(2) to each term in the summation over $n$, we have
$$ S_3 = \sum_{R \geq 0} \frac{u^{2R} (q^{2R} + 1) (-uq^{4R-2}; 1/q^2)_{\infty}}{4q^{4R^2} (1/q^2; 1/q^2)_{2R}},$$
and from the identity
$$ (-uq^{4R-2}; 1/q^2)_{\infty} = u^{2R} q^{2R(2R-1)} (-1/u; 1/q^2)_{2R} (-u/q^2; 1/q^2)_{\infty},$$
we have
\begin{equation} \label{S3}
S_3 = \frac{(-u/q^2; 1/q^2)_{\infty} }{4} \sum_{R \geq 0} \frac{(u^2)^{2R} (-1/u; 1/q^2)_{2R} + (u^2/q)^{2R} (-1/u; 1/q^2)_{2R}}{(1/q^2; 1/q^2)_{2R}}.
\end{equation}
Next we consider
\begin{align*}
T_3 :& = \sum_{n \geq 0} u^n q^{n^2} \sum_{R=1 \atop{R \text{ odd}}}^{n} \frac{1}{|\gO^-(2R,q)||\gO(2n+1-2R,q)|} \\
& = \sum_{n \geq 1} \sum_{R=0}^{\lfloor (n-1)/2 \rfloor} \frac{u^n q^{n^2}}{ |\gO^-(4R+2,q)| |\gO(2n-1 - 4R, q)| } \\
& = \sum_{R \geq 0} \sum_{n \geq 1} \frac{u^n q^{n^2} (q^{2R+1} - 1)}{ 4q^{8R^2} (1/q^2 ; 1/q^2)_{2R} \, q^{2(n-2R)^2 + (n-2R)} (1/q^2 ; 1/q^2)_{n-2R}}.
\end{align*}
We replace $n$ with $n+2R+1$, and rewrite the expression to obtain
$$T_3 = \sum_{R \geq 0} \frac{u^{2R+1}}{4q^{(2R+1)^2} (1/q^2 ; 1/q^2)_{2R+1}} \sum_{n \geq 0} \frac{q^{2R+1}(uq^{4R})^n - (uq^{4R})^n}{q^{2\binom{n}{2}} (1/q^2 ; 1/q^2)_n}.$$
Apply Lemma \ref{qbinom}(2) again for each term in the summation over $n$ to get
$$T_3 = \sum_{R \geq 0} \frac{u^{2R+1}(q^{2R+1} - 1) (-uq^{4R}; 1/q^2)_{\infty}}{4 q^{(2R+1)^2} (1/q^2 ; 1/q^2)_{2R+1}}.$$
Now use the identity
$$ (-uq^{4R}, 1/q^2)_{\infty} = u^{2R+1} q^{2R(2R+1)} (-1/u ; 1/q^2)_{2R+1} (-u/q^2; 1/q^2)_{\infty}$$
to rewrite $T_3$ as
\begin{equation} \label{T3}
T_3 = \frac{(-u/q^2; 1/q^2)_{\infty}}{4} \sum_{R \geq 0} \frac{(u^2)^{2R+1} (-1/u; 1/q^2)_{2R+1} + (-u^2/q)^{2R+1} (-1/u; 1/q^2)_{2R+1}}{(1/q^2; 1/q^2)_{2R+1}}.
\end{equation}
From \eqref{S3} and \eqref{T3}, we have
$$S_3 + T_3 = \frac{(-u/q^2; 1/q^2)_{\infty}}{4} \sum_{R \geq 0} \frac{(u^2)^{R} (-1/u; 1/q^2)_{R} + (-u^2/q)^{R} (-1/u; 1/q^2)_{R}}{(1/q^2; 1/q^2)_{R}}.$$
Apply Lemma \ref{qbinom}(1) $y=1/q^2$, $A = -1/u$, and $x = -u^2, -u^2/q$ to the terms in the sum above to obtain
$$S_3 + T_3 = \frac{1}{4}\left [ \frac{(-u/q^2 ; 1/q^2)_{\infty}(u; 1/q^2)_{\infty}}{(-u^2 ; 1/q^2)_{\infty}} + \frac{(-u/q^2; 1/q^2)_{\infty} (u/q; 1/q^2)_{\infty} }{(-u^2/q; 1/q^2)_{\infty}}\right],$$
for $|u| < 1$, which simplifies to the desired generating function.
\end{proof}

\subsection{Groups over fields of characteristic two}

\begin{theorem} \label{GFOmqeven}
  For $q$ even and $|u| < 1/q$, we have the generating function
$$ \displaystyle \sum_{n \geq 0 }  \frac{i(\Om^{\pm}(2n, q))}{|\gO^{\pm}(2n,q)|}  q^{n^2} u^n = \frac{1}{2} \left[ \frac{ \prod_{i \geq 1} (1 + u/q^{2(i-1)}) } {  \prod_{i \geq 1} (1 - u^2/q^{2(i-2)} ) } \pm \frac{ \prod_{i \geq 1} (1 + u/q^{2i-1}) } {  \prod_{ i \geq 1} (1 - u^2/q^{2(i-1)} ) } \right].$$
\end{theorem}
\begin{proof}  Let $A_k^{\pm}$ and $B_k$ be as in Proposition \ref{Oeven}.  We use the computations made in \cite[Proof of Theorem 2.18]{FuGuSt16}.  We have
$$ \sum_{n \geq 0} u^n q^{n^2} \sum_{k=2 \atop{k \text{ even}}}^{n} \frac{1}{B_k} = \sum_{n \geq 0} u^n q^{n^2} \sum_{k=0 \atop{k \text{ even}}}^{n} \frac{1}{2q^{k(k+1)/2 + (k-1)(2n-2k)-1} |\Sp(k-2,q)| |\Sp(2n-2k,q)|}.$$
By substituting $2R+2$ for $k$ and using $|\Sp(2n,q)| = q^{2n^2 + n} (1/q^2 ; 1/q^2)_n$, an application of \cite[Lemma 2.5]{FuGuSt16} yields
\begin{align*}
\sum_{n \geq 0} u^n q^{n^2} &\sum_{k=2 \atop{k \text{ even}}}^{n} \frac{1}{B_k} \\
& = \frac{q^4}{2} \sum_{n \geq 0} u^n q^{n^2} \sum_{R=0}^{\lfloor (n-2)/2 \rfloor} \frac{q^{8R-2n}}{q^{4nR-6R^2} q^{2R^2} (1/q^2; 1/q^2)_R \, q^{2(n-2R-2)^2} (1/q^2; 1/q^2)_{n-2R-2}}\\
& = \frac{u^2 q^2}{2} \frac{(-u; 1/q^2)_{\infty}}{(u^2 q^2 ; 1/q^2)_{\infty}} = \frac{u^2 q^2}{2} \frac{ \prod_{i \geq 1} (1 + u/q^{2(i-1)}) } { \prod_{i \geq 1} (1 - u^2/q^{2(i-2)})}.
\end{align*}
Next, we have
$$ \sum_{n \geq 0} u^n q^{n^2} \sum_{k = 0 \atop{k \text{ even}}}^{n} \frac{1}{A_k^+} = \sum_{n \geq 0} u^n q^{n^2} \sum_{k=0 \atop{k \text{ even}}}^{n} \frac{1}{q^{k(k-1)/2 + k(2n-2k)} |\Sp(k,q)| |\gO^+(2n-2k,q)|}.$$
Again, following \cite[Proof of Theorem 2.18]{FuGuSt16}, we use that $|\gO^+(2n,q)| = \frac{2q^{2n^2}}{q^n + 1} (1/q^2 ; 1/q^2)_n$, substitute $2R$ for $k$ to obtain
$$\sum_{n \geq 0} u^n q^{n^2}  \sum_{k = 0 \atop{k \text{ even}}}^{n} \frac{1}{A_k^+} 
 = \frac{1}{2} \sum_{n \geq 0}  u^n q^{n^2 - n} \sum_{R=0}^{\lfloor n/2 \rfloor} \frac{q^{2n - 2R} + q^n}{ q^{4nR-6R^2} q^{2R^2} (1/q^2 ; 1/q^2)_R \, q^{2(n-2R)^2} (1/q^2 ; 1/q^2)_{n-2R}}.$$
Applying \cite[Lemma 2.5]{FuGuSt16} twice, once for each term in the numerator, gives us
\begin{align*}
\sum_{n \geq 0} u^n q^{n^2} & \sum_{k = 0 \atop{k \text{ even}}}^{n} \frac{1}{A_k^+} = \frac{1}{2} \frac{(-u; 1/q^2)_{\infty}}{(u^2; 1/q^2)_{\infty}} + \frac{1}{2} \frac{(-u/q; 1/q^2)_{\infty}}{(u^2 ; 1/q^2)_{\infty}}\\
& = \frac{1}{2} \frac{ \prod_{i \geq 1} (1 + u/q^{2(i-1)})}{\prod_{i \geq 1} (1 - u^2/q^{2(i-1)})} + \frac{1}{2} \frac{ \prod_{i \geq 1} (1 + u/q^{2i-1})} { \prod_{i \geq 1} (1 - u^2/q^{2(i-1)})}.
\end{align*}
For the case of $A_k^-$ terms, we use the convention that $1/|\gO^-(0,q)| = 0$ so that we can apply the same computations as above, which gives
\begin{align*}
\sum_{n \geq 0} u^n & q^{n^2}   \sum_{k = 0 \atop{k \text{ even}}}^{n-1} \frac{1}{A_k^-} = \sum_{n \geq 0} u^n q^{n^2}  \sum_{k = 0 \atop{k \text{ even}}}^{n} \frac{1}{A_k^-} \\
& = \frac{1}{2} \sum_{n \geq 0}  u^n q^{n^2 - n} \sum_{R=0}^{\lfloor n/2 \rfloor} \frac{q^{2n - 2R} - q^n}{ q^{4nR-6R^2} q^{2R^2} (1/q^2 ; 1/q^2)_R \, q^{2(n-2R)^2} (1/q^2 ; 1/q^2)_{n-2R}} \\
& = \frac{1}{2} \frac{(-u; 1/q^2)_{\infty}}{(u^2; 1/q^2)_{\infty}} - \frac{1}{2} \frac{(-u/q; 1/q^2)_{\infty}}{(u^2 ; 1/q^2)_{\infty}} = \frac{1}{2} \frac{ \prod_{i \geq 1} (1 + u/q^{2(i-1)})}{\prod_{i \geq 1} (1 - u^2/q^{2(i-1)})} - \frac{1}{2} \frac{ \prod_{i \geq 1} (1 + u/q^{2i-1})} { \prod_{i \geq 1} (1 - u^2/q^{2(i-1)})}.
\end{align*}
From all of the above, Proposition \ref{Omeven}(1) and (2), and the fact that
\begin{align*}
\frac{u^2 q^2}{2} & \frac{ \prod_{i \geq 1} (1 + u/q^{2(i-1)}) } { \prod_{i \geq 1} (1 - u^2/q^{2(i-2)})} + \frac{1}{2} \frac{ \prod_{i \geq 1} (1 + u/q^{2(i-1)})}{\prod_{i \geq 1} (1 - u^2/q^{2(i-1)})} \\
& = \frac{u^2 q^2}{2} \frac{ \prod_{i \geq 1} (1 + u/q^{2(i-1)}) } { \prod_{i \geq 1} (1 - u^2/q^{2(i-2)})} + \frac{1 - u^2 q^2 }{2} \frac{ \prod_{i \geq 1} (1 + u/q^{2(i-1)})}{\prod_{i \geq 1} (1 - u^2/q^{2(i-2)})} = \frac{1}{2} \frac{ \prod_{i \geq 1} (1 + u/q^{2(i-1)})}{\prod_{i \geq 1} (1 - u^2/q^{2(i-2)})},
\end{align*}
the claimed generating function follows.
\end{proof}

We note the results in Theorem \ref{OOmqeven} hold for all $n$ if and only if the character degree sum of $\Om^{\pm}(4m,q)$ is $|\gO^{\pm}(4m,q)|/q^{4m^2}$ times the coefficient of $u^{2m}$ of the generating function in Theorem \ref{GFOmqeven}, and the character degree sum of $\Om^{\pm}(4m+2,q)$ is $|\gO^{\pm}(4m+2,q)|/q^{(2m+1)^2}$ times the coefficient of $u^{2m+1}$ in the following generating function.

\begin{theorem} \label{GFO-Omqeven}
For $q$ even and $|u| < 1/q$, we have the generating function
$$\sum_{n \geq 0} \frac{i(\gO^{\pm}(2n,q) \setminus \Om^{\pm}(2n,q))}{|\gO^{\pm}(2n,q)|}q^{n^2} u^n = \frac{uq}{2} \frac{ \prod_{i \geq 1} (1 + u/q^{2i-1}) }{ \prod_{i \geq 1} (1 - u^2/q^{2(i-2)}) }.$$
\end{theorem}
\begin{proof} From \cite[Theorems 2.18 and 2.19]{FuGuSt16} and Proposition \ref{Oeven}, we have for $q$ even and $|u| < 1/q$,
$$\sum_{n \geq 0} \frac{i(\gO^{\pm}(2n,q))}{|\gO^{\pm}(2n,q)|} q^{n^2} u^n = \frac{1}{2(1-uq)} \frac{\prod_{i \geq 1} (1 + u/q^{2(i-1)})}{\prod_{i \geq 1} (1 -u^2/q^{2(i-1)})} \pm \frac{1}{2} \frac{ \prod_{i \geq 1} (1 + u/q^{2i-1})}{ \prod_{i \geq 1} (1 - u^2/q^{2(i-1)})}.$$
The result is obtained by taking the difference of this and the generating function in Theorem \ref{GFOmqeven}.
\end{proof}

\section{Asymptotics on Involutions} \label{Asymptotics}

In this section, we compute the asymptotic behavior of the number of involutions in the groups $\SO^{\pm}(n,q)$ and $\Om^{\pm}(n,q)$ when $q$ is fixed and $n$ grows.  The main tool we use  is the following result of Darboux, which is the same method used in \cite{FuGuSt16}, and which may be found in \cite{Odl}.  The idea is to treat the generating functions in the previous section as analytic functions, and then equate coefficients of power series.  If $f(u)$ is a power series in $u$, then we denote by $[u^n]f(u)$ the coefficient of $u^n$ in this power series.

\begin{lemma} \label{darboux}
Suppose for some $r>0$ that $f(u)$ is analytic on $|u|<r$ and has a finite number of simple poles on $|u|=r$, denoted by $w_j$.  Suppose we may write $f(u) = \sum_j \frac{g_j(u)}{1 - u/w_j}$ with $g_j(u)$ analytic in a neighborhood of $w_j$.  Then we have
$$ [u^n] f(u) = \sum_j \frac{g_j(w_j)}{w_j^n} + o(1/r^n).$$
\end{lemma}

The following result is implicit in \cite{FuGuSt16}, and follows directly from Lemma \ref{darboux}.  We state it here for convenience, as it is used repeatedly. 

\begin{corollary} \label{darbouxcor}
If $f$ satisfies the assumptions of Lemma \ref{darboux} with $r > 1$, then
$$\lim_{n \rightarrow \infty} [u^n]f(u) = 0.$$
\end{corollary}

Since the proofs of the results in this section are very similar to each other, we give all details in the first proof, and give fewer details for the proofs that follow.  

\subsection{Groups over fields of odd characteristic}

\begin{theorem} \label{AsSOqodd} If $q$ is odd and fixed, then
$$ \lim_{m \rightarrow \infty} \frac{i(\SO^{\pm}(4m,q))}{q^{(2m)^2}} = \frac{1}{2} \left( \prod_{i \geq 1} (1 + 1/q^{2i-1})^2 + \prod_{i \geq 1} (1 - 1/q^{2i-1})^2 \right),$$
and
$$ \lim_{m \rightarrow \infty} \frac{i(\SO^{\pm}(4m+2,q))}{q^{(2m+1)^2}} = \frac{1}{2} \left( \prod_{i \geq 1} (1 + 1/q^{2i-1})^2 - \prod_{i \geq 1} (1 - 1/q^{2i-1})^2 \right).$$
\end{theorem}
\begin{proof}  First replace $u$ by $u/q$ in the generating function in Theorem \ref{GFSOqodd} to obtain
\begin{align}
\sum_{n \geq 0}  \frac{i(\SO^{\pm}(2n,q))}{|\gO^{\pm}(2n,q)|}  q^{n^2 - n} u^n & = \sum_{n \geq 0} u^n \frac{i(\SO^{\pm}(2n,q))}{2 q^{n^2} (1 \mp 1/q^n)(1 - 1/q^2) \cdots (1 - 1/q^{2(n-1)})} \label{leftside1}\\
& =  \frac{1}{2} \left[ \frac{ \prod_{i \geq 1} (1 + u/q^{2i-1})^2 }{\prod_{i \geq 1} (1 - u^2/q^{2(i-1)} ) } \pm \frac{ \prod_{i \geq 1} (1 + u/q^{2i})^2 }{ \prod_{i \geq 1} (1 - u^2/q^{2i}) } \right] \label{rightside1}
\end{align}
We now consider the limit as $n \rightarrow \infty$ of the coefficient of $u^n$ in both \eqref{leftside1} and \eqref{rightside1}.  In \eqref{rightside1}, first note that the second quotient of products is analytic for $|u|<q$, where $q > 1$.  It follows from Corollary \ref{darbouxcor} that the limit as $n \rightarrow \infty$ of the coefficient of $u^n$ in this expression is $0$.  The first quotient of products in \eqref{rightside1} is analytic in the unit disk, and has simple poles at $u=1$ and $u=-1$ (which we take as $w_1$ and $w_2$ in applying Lemma \ref{darboux} here).  By factoring out $1/(1-u^2)$, we may rewrite this quotient of products as 
$$ \frac{1}{4} \frac{ \prod_{i \geq 1} (1 + u/q^{2i-1})^2 }{\prod_{i \geq 1} (1 - u^2/q^{2i} ) } \left[ \frac{1}{1+u} + \frac{1}{1-u} \right],$$
where the quotient of products outside of the brackets is analytic in neighborhoods of $1$ and $-1$ (and this expression is $g_1(u)$ and $g_2(u)$ in Lemma \ref{darboux}).  By Lemma \ref{darboux}, the coefficient of $u^n$ is
$$\frac{1}{4} \left( \frac{ \prod_{i \geq 1} (1 + 1/q^{2i-1})^2 + (-1)^n \prod_{i \geq 1} (1 - 1/q^{2i-1})^2 }{ \prod_{i \geq 1} (1 - 1/q^{2i}) } \right) + o(1).$$
If we take either $n=2m$ or $n=2m+1$, this gives the limit as $m \rightarrow \infty$ of the coefficient of $u^n$ in \eqref{rightside1} to be
\begin{equation} \label{rightside1lim}
\frac{1}{4} \left( \frac{ \prod_{i \geq 1} (1 + 1/q^{2i-1})^2 \pm \prod_{i \geq 1} (1 - 1/q^{2i-1})^2 }{ \prod_{i \geq 1} (1 - 1/q^{2i}) } \right),
\end{equation}
where we take $+$ if $n=2m$ and $-$ if $n=2m+1$.

If we take either $n=2m$ or $n = 2m+1$, we compute the limit of the coefficient of $u^n$ in \eqref{leftside1} to be
\begin{align} 
\lim_{m \rightarrow \infty} \frac{i(\SO^{\pm}(2n,q))}{2 q^{n^2} (1 \mp 1/q^n)(1 - 1/q^2) \cdots (1 - 1/q^{2(n-1)})}& = \frac{1}{2 \prod_{i \geq 1} (1 - 1/q^{2i})} \lim_{m \rightarrow \infty} \frac{i(\SO^{\pm}(2n,q))}{q^{n^2}(1 \mp 1/q^n)} \nonumber \\
= \frac{1}{2 \prod_{i \geq 1} (1 - 1/q^{2i})} \lim_{m \rightarrow \infty} \frac{i(\SO^{\pm}(2n,q))}{q^{n^2}} \label{leftside1lim}
\end{align}
Equating the expressions in \eqref{rightside1lim} and \eqref{leftside1lim} gives the result.
\end{proof}

From Theorem \ref{AsSOqodd} and the asymptotics for the number of involutions in $\Sp(2n,q)$ computed in \cite[Theorem 5.2]{FuGuSt16}, we have the following interesting result that the number of involutions in $\SO^{\pm}(2n,q)$ and $\Sp(2n,q)$ are asymptotically equal for $q$ odd, although the numbers of involutions in these groups are not equal in general.  We note that in the case that $n$ is odd, this result follows by dividing both sides of \cite[Theorem 5.2(2)]{FuGuSt16} by $q$.

\begin{corollary} \label{SOSpsame}
Let $q$ be odd and fixed.  Then
$$ \lim_{n \rightarrow \infty} \frac{i(\SO^{\pm}(2n,q))}{i(\Sp(2n,q))} = 1.$$
\end{corollary}

The following is a quick result, and we include it for the sake of completeness.

\begin{theorem} \label{AsSOnodd} If $q$ is odd and fixed, then
$$ \lim_{n \rightarrow \infty} \frac{i(\SO(2n+1,q))}{q^{n^2 + n}} = \prod_{i \geq 1} (1 + 1/q^{2i})^2.$$
\end{theorem}
\begin{proof} We have $i(\SO(2n+1,q)) = (1/2) i(\gO(2n+1,q))$ from Proposition \ref{Oodd}(3).  We also have from \cite[Theorem 6.3]{FuGuSt16} that
$$ \lim_{n \rightarrow \infty} \frac{i(\gO(2n+1,q))}{q^{n^2 + n}} = 2 \prod_{i \geq 1} (1 + 1/q^{2i})^2,$$
and the result follows.
\end{proof}

We note the sign switch in the following result, in comparison with Theorem \ref{AsSOqodd}, and its relation to indicators of these groups in Theorem \ref{SOtwist}.

\begin{theorem} \label{AsO-SOqodd} If $q$ is odd and fixed, then
$$ \lim_{m \rightarrow \infty} \frac{i(\gO^{\pm}(4m,q) \setminus \SO^{\pm}(4m,q))}{q^{(2m)^2}} = \frac{1}{2} \left( \prod_{i \geq 1} (1 + 1/q^{2i-1})^2 - \prod_{i \geq 1} (1 - 1/q^{2i-1})^2 \right),$$
and
$$ \lim_{m \rightarrow \infty} \frac{i(\gO^{\pm}(4m+2, q) \setminus \SO^{\pm}(4m+2,q))}{q^{(2m+1)^2}} = \frac{1}{2} \left( \prod_{i \geq 1} (1 + 1/q^{2i-1})^2 + \prod_{i \geq 1} (1 - 1/q^{2i-1})^2 \right).$$
\end{theorem}
\begin{proof} This follows from Theorem \ref{AsSOqodd} and the result \cite[Theorem 6.2]{FuGuSt16} that when $q$ is odd and fixed,
$$ \lim_{n \rightarrow \infty} \frac{i(\gO^{\pm}(2n,q))}{q^{n^2}} = \prod_{i \geq 1} (1 + 1/q^{2i-1})^2,$$
by taking the difference.
\end{proof}

In the next two results, we note that the asymptotics we obtain are independent of $q$ mod $4$ and the type of defining quadratic form, even though the generating functions for these cases are quite different.

\begin{theorem} \label{AsOmqodd} If $q$ is odd and fixed, then
$$ \lim_{m \rightarrow \infty} \frac{i(\Om^{\pm}(4m,q))}{q^{(2m)^2}} = \frac{1}{4} \left( \prod_{i \geq 1} (1 + 1/q^{2i-1})^2 + \prod_{i \geq 1} (1 - 1/q^{2i-1})^2 \right),$$
and
$$ \lim_{m \rightarrow \infty} \frac{i(\Om^{\pm}(4m+2,q))}{q^{(2m+1)^2}} = \frac{1}{4} \left( \prod_{i \geq 1} (1 + 1/q^{2i-1})^2 - \prod_{i \geq 1} (1 - 1/q^{2i-1})^2 \right).$$
\end{theorem}
\begin{proof} We first replace $u$ by $u/q$ in each of the generating functions in Theorems \ref{GFOmq1mod4}(1), \ref{GFOmq1mod4}(2), and \ref{GFOmq3mod4}(1).  After this substitution, all products of quotients have poles with magnitude greater than $1$ (with the largest in magnitude of the others being $\pm q$, $\pm \sqrt{q}$, and $\pm \sqrt{-q}$), except for the term
$$ \frac{1}{4} \frac{\prod_{i \geq} (1 + u/q^{2i-1})^2}{\prod_{i \geq 1} (1 - u^2/q^{2(i-1)})} = \frac{1}{8} \frac{ \prod_{i \geq 1} (1 + u/q^{2i-1})^2 }{\prod_{i \geq 1} (1 - u^2/q^{2i} ) } \left[ \frac{1}{1+u} + \frac{1}{1-u} \right],$$
which is analytic on $|u|<1$.  If $f(u)$ is any of the generating functions of interest, then we obtain from Lemma \ref{darboux} and Corollary \ref{darbouxcor} that if $n=2m$ or $n=2m+1$, then
\begin{equation} \label{rightside2lim}
\lim_{m \rightarrow \infty} [u^n]f(u) = \frac{1}{8} \left( \frac{ \prod_{i \geq 1} (1 + 1/q^{2i-1})^2 \pm \prod_{i \geq 1} (1 - 1/q^{2i-1})^2 }{ \prod_{i \geq 1} (1 - 1/q^{2i}) } \right),
\end{equation}
where we take $+$ if $n=2m$ and $-$ if $n = 2m+1$.

This expression is equal to the coefficient of $u^n$ as $n \rightarrow \infty$ of 
\begin{equation} 
\sum_{n \geq 0}  \frac{i(\Om^{\pm}(2n,q))}{|\gO^{\pm}(2n,q)|}  q^{n^2 - n} u^n  = \sum_{n \geq 0} u^n \frac{i(\Om^{\pm}(2n,q))}{2 q^{n^2} (1 \mp 1/q^n)(1 - 1/q^2) \cdots (1 - 1/q^{2(n-1)})}.
\end{equation}
Taking $n=2m$ or $n=2m+1$, the limit of the coefficient of $u^n$ as $m \rightarrow \infty$ of this is
\begin{equation} \label{leftside2lim}
\frac{1}{2 \prod_{i \geq 1} (1 - 1/q^{2i})} \lim_{m \rightarrow \infty} \frac{i(\SO^{\pm}(2n,q))}{q^{n^2}}.
\end{equation}
Equating \eqref{rightside2lim} and \eqref{leftside2lim} gives the result.
\end{proof}

\begin{theorem} \label{AsOmnodd} If $q$ is odd and fixed, then
$$ \lim_{n \rightarrow \infty} \frac{i(\Om(2n+1,q))}{q^{n^2 + n}} = \frac{1}{2} \prod_{i \geq 1} (1 + 1/q^{2i})^2.$$
\end{theorem}
\begin{proof}
We consider the power series
\begin{equation} \label{leftside4}
\sum_{n \geq 0}  \frac{i(\Om(2n+1,q))}{|\gO(2n+1,q)|}  q^{n^2} u^n  = \sum_{n \geq 0} u^n \frac{i(\Om(2n+1,q))}{2 q^{n^2+n} (1 - 1/q^2) \cdots (1 - 1/q^{2n})}, 
\end{equation}
which is also given by the generating functions in Theorems \ref{GFOmq1mod4}(3) and \ref{GFOmq3mod4}(2).  In both of those generating functions, the first product of quotients term, given by
$$ \frac{1}{4} \frac{ (1 + u) \prod_{i \geq 1} (1 + u/q^{2i})^2}{ \prod_{i \geq 1} (1 - u^2/q^{2(i-1)})} = \frac{1}{4} \frac{1}{1-u} \frac{ \prod_{i \geq 1} (1 + u/q^{2i})^2}{ \prod_{i \geq 1} (1 - u^2/q^{2i})},$$
has a simple pole at $u=1$ and is analytic in $|u|<1$.  The second product of quotients in each of those generating functions has pole of smallest magnitude given by $\pm \sqrt{q}$ and $\pm \sqrt{-q}$, respectively.  It follows from Lemma \ref{darboux} and Corollary \ref{darbouxcor} that the limit as $n \rightarrow \infty$ of the coefficient of $u^n$ in each of these generating functions is given by 
\begin{equation} \label{rightside4lim}
\frac{1}{4} \frac{ \prod_{i \geq 1} (1 + 1/q^{2i})^2}{ \prod_{i \geq 1} (1 - 1/q^{2i})}.
\end{equation}
Meanwhile, the limit as $n \rightarrow \infty$ of the coefficient of $u^n$ in \eqref{leftside4} is
\begin{equation} \label{leftside4lim}
\frac{1}{2 \prod_{i \geq 1} (1 - q^{2i})} \lim_{n \rightarrow \infty} \frac{i(\Om(2n+1,q))}{q^{n^2 + n}}.
\end{equation}
Equating \eqref{rightside4lim} and \eqref{leftside4lim} gives the result.
\end{proof}

While we do have $i(\Om^-(2n,q)) = (1/2) i(\SO^-(2n,q))$ when $q \equiv 1($mod $4)$, in general $\Om^{\pm}(n,q)$ does not have half of the number of involutions of $\SO^{\pm}(n,q)$.  However, we do have this asymptotically in the following, which follows from Theorems \ref{AsSOqodd}, \ref{AsSOnodd}, \ref{AsOmqodd}, and \ref{AsOmnodd}.

\begin{corollary} \label{onehalf}
For any fixed odd $q$, we have
$$ \lim_{n \rightarrow \infty} \frac{i(\Om^{\pm}(n,q))}{i(\SO^{\pm}(n,q))} = \frac{1}{2}.$$
\end{corollary}

One might expect for a similar result to hold for $i(\SO^{\pm}(n,q))/i(\gO^{\pm}(n,q))$.  However, it follows from Theorem \ref{AsSOqodd} and \cite[Theorem 6.2]{FuGuSt16} that when $n$ is even, the limit of this expression is dependent on $q$.

\subsection{Groups over fields of characteristic two}

In our last two asymptotic results, we note the similarity to the results in Theorems \ref{AsSOqodd} and \ref{AsO-SOqodd}.

\begin{theorem} \label{AsOmqeven} If $q$ is even and fixed, then
$$ \lim_{m \rightarrow \infty} \frac{i(\Om^{\pm}(4m,q))}{q^{(2m)^2}} = \frac{1}{2} \left( \prod_{i \geq 1} (1 + 1/q^{2i-1}) + \prod_{i \geq 1} (1 - 1/q^{2i-1}) \right),$$
and
$$ \lim_{m \rightarrow \infty} \frac{i(\Om^{\pm}(4m+2,q))}{q^{(2m+1)^2}} = \frac{1}{2} \left( \prod_{i \geq 1} (1 + 1/q^{2i-1}) - \prod_{i \geq 1} (1 - 1/q^{2i-1}) \right).$$
\end{theorem}
\begin{proof} We replace $u$ by $u/q$ in Theorem \ref{GFOmqeven} to obtain
\begin{align}
\sum_{n \geq 0}  \frac{i(\Om^{\pm}(2n,q))}{|\gO^{\pm}(2n,q)|} q^{n^2 - n} u^n  & = \sum_{n \geq 0} u^n \frac{i(\Om^{\pm}(2n,q))}{2q^{n^2} (1 \mp 1/q^n)(1 - 1/q^2) \cdots (1 - 1/q^{2(n-1)})} \label{leftside5}\\
& =  \frac{1}{2} \left[ \frac{ \prod_{i \geq 1} (1 + u/q^{2i-1}) }{\prod_{i \geq 1} (1 - u^2/q^{2(i-1)} ) } \pm \frac{ \prod_{i \geq 1} (1 + u/q^{2i}) }{ \prod_{i \geq 1} (1 - u^2/q^{2i}) } \right] \label{rightside5}
\end{align}
The second quotient of products in \eqref{rightside5} is analytic when $|u|<q$, where $q > 1$, and so by Corollary \ref{darbouxcor} the limit of the coefficient of $u^n$ in this expression as $n \rightarrow \infty$ is $0$.  The first quotient of products is analytic in the unit disk with simple poles at $u=1$ and $u=-1$.  We rewrite this quotient of products as 
$$ \frac{1}{4} \frac{ \prod_{i \geq 1} (1 + u/q^{2i-1}) }{\prod_{i \geq 1} (1 - u^2/q^{2i} ) } \left[ \frac{1}{1+u} + \frac{1}{1-u} \right],$$
where the quotient of products outside of the brackets is analytic in neighborhoods of $1$ and $-1$.  If we take either $n=2m$ or $n=2m+1$, then by Lemma \ref{darboux}, the limit as $m \rightarrow \infty$ of the coefficient of $u^n$ in \eqref{rightside5} to be
\begin{equation} \label{rightside5lim}
\frac{1}{4} \left( \frac{ \prod_{i \geq 1} (1 + 1/q^{2i-1}) \pm \prod_{i \geq 1} (1 - 1/q^{2i-1}) }{ \prod_{i \geq 1} (1 - 1/q^{2i}) } \right),
\end{equation}
where we take $+$ if $n=2m$ and $-$ if $n=2m+1$.

If we take either $n=2m$ or $n = 2m+1$, we compute the limit of the coefficient of $u^n$ in \eqref{leftside5} to be
\begin{equation} \label{leftside5lim}
\frac{1}{2 \prod_{i \geq 1} (1 - 1/q^{2i})} \lim_{m \rightarrow \infty} \frac{i(\Om^{\pm}(2n,q))}{q^{n^2}}.
\end{equation}
Equating the expressions in \eqref{rightside5lim} and \eqref{leftside5lim} gives the result.
\end{proof}

\begin{theorem} If $q$ is even and fixed, then
$$ \lim_{m \rightarrow \infty} \frac{i(\gO^{\pm}(4m,q) \setminus \Om^{\pm}(4m,q))}{q^{(2m)^2}} = \frac{1}{2} \left( \prod_{i \geq 1} (1 + 1/q^{2i-1}) - \prod_{i \geq 1} (1 - 1/q^{2i-1}) \right),$$
and
$$ \lim_{m \rightarrow \infty} \frac{i(\gO^{\pm}(4m+2,q) \setminus \Om^{\pm}(4m+2,q))}{q^{(2m+1)^2}} = \frac{1}{2} \left( \prod_{i \geq 1} (1 + 1/q^{2i-1}) + \prod_{i \geq 1} (1 - 1/q^{2i-1}) \right).$$
\end{theorem}
\begin{proof} This follows directly from Theorem \ref{AsOmqeven}, and the result \cite[Theorem 6.7]{FuGuSt16} that for $q$ even and fixed,
$$ \lim_{n \rightarrow \infty} \frac{i(\gO^{\pm}(2n,q))}{q^{n^2}} = \prod_{i \geq 1} (1 + 1/q^{2i-1}),$$
by taking the difference.
\end{proof}

\newpage

\appendix
\section*{Appendix}
\renewcommand{\thesubsection}{\Alph{subsection}}


The following table gives the number of involutions (using \cite[Theorem 5.3]{FuGuSt16}) and the character degree sum (using \cite{LuWWW}) for $\Sp(2n,q)$ with $q$ even and $2 \leq n \leq 8$.  These computations complete the proof of Theorem \ref{Spqeven}.

\bigskip

\noindent
\renewcommand*{\arraystretch}{1.2}
\begin{tabular}[H!]{| m{4.8cm} | m{5cm} | m{4.8cm} |}
\hline
\begin{center}\textbf{Character degree sum of}\end{center} & \begin{center}\textbf{Number of involutions in}\end{center} & \begin{center} \textbf{Expression} \end{center} \\ 
\hline
$\Sp(4,q)$  &  $\Sp(4,q)$   &  $q^6+q^4-q^2$ \\
\hline
$\Sp(6,q)$  &  $\Sp(6,q)$   &  $q^{12}+q^{10}-q^4$ \\
\hline
$\Sp(8,q)$  &  $\Sp(8,q)$   &  $q^{20}+q^{18}+q^{16}-q^{12}-q^{10}$ \\
\hline
$\Sp(10,q)$  &  $\Sp(10,q)$   &  $q^{30}+q^{28}+q^{26}+q^{24}-q^{20}-q^{18}-q^{16}-q^{14}+q^{10}$ \\
\hline
$\Sp(12,q)$  &  $\Sp(12,q)$   &  $q^{42}+q^{40}+q^{38}+2q^{36}-q^{30}-q^{28}-2q^{26}-q^{24}+q^{14}$ \\
\hline
$\Sp(14,q)$  &  $\Sp(14,q)$   &  $q^{56}+q^{54}+q^{52}+2q^{50}+q^{48}+q^{46}-q^{42}-2q^{40}-2q^{38}-2q^{36}-q^{34}-q^{32}+q^{28}+q^{26}+q^{24}$ \\
\hline
$\Sp(16,q)$  &  $\Sp(16,q)$   &  $q^{72}+q^{70}+q^{68}+2q^{66}+2q^{64}+q^{62}+q^{60}-q^{56}-2q^{54}-2q^{52}-3q^{50}-2q^{48}-2q^{46}-q^{44}+q^{40}+q^{38}+q^{36}+q^{34}+q^{32}+q^{30}-q^{24}$ \\
\hline
\end{tabular}

\bigskip

\noindent
The following table gives the number of involutions in either $\Om^{\pm}(2n,q)$ or $\gO^{\pm}(2n,q) \setminus \Om^{\pm}(2n,q)$ (using Proposition \ref{Omeven}), and the character degree sum (using \cite{LuWWW}) of $\Om^{\pm}(2n,q)$, for $q$ even and the values of $n$ needed for the proof of Theorem \ref{OOmqeven}.

\bigskip

\noindent
\renewcommand*{\arraystretch}{1.2}
\begin{tabular}[H!]{| m{4.8cm} | m{5cm} | m{4.8cm} |}
\hline
\begin{center}\textbf{Character degree sum of} \end{center}& \begin{center}\textbf{Number of involutions in} \end{center}         & \begin{center} \textbf{Expression} \end{center}         \\ 
\hline
$\Omega^\pm(4,q)$ & $\Omega^\pm(4,q)$ & $q^4$ \\
\hline
$\Omega^\pm(6,q)$ & $\gO^\pm(6,q) \setminus \Omega^\pm(6,q)$ & $q^9 \mp q^6$  \\ 
\hline
$\Omega^\pm(8,q)$ & $ \Omega^\pm(8,q)$ & $q^{16}+q^{12}-q^4$  \\ \hline
$\Omega^\pm(10,q)$ & $\gO^\pm(10,q) \setminus \Omega^\pm(10,q)$ & $q^{25}+q^{21}\mp q^{20}\mp q^{16}-q^{13}\pm q^8$  \\ 
\hline
$\Omega^\pm(12,q)$ & $\Omega^\pm(12,q)$ & $q^{36}+q^{32}+q^{30}+q^{28}-q^{22}-q^{20}-q^{18}-q^{16}+q^{10}$  \\ 
\hline
$\Omega^\pm(14,q)$ & $\gO^\pm(14,q) \setminus \Omega^\pm(14,q)$ & $q^{49}+q^{45}+q^{43}\mp q^{42}+q^{41}\mp q^{38}\mp q^{36}-q^{35}\mp q^{34}-q^{33}-q^{31}-q^{29}\pm q^{28}\pm q^{26}\pm q^{24}+q^{23}\pm q^{22}\mp q^{16}$ \\
\hline
$\Omega^+(16,q)$ & $\Omega^+(16,q)$ & $q^{64}+q^{60}+q^{58}+2q^{56}+q^{54}+q^{52}-q^{46}-2q^{44}-2q^{42}-2q^{40}-q^{38}-q^{36}+q^{30}+q^{28}+q^{26}$ \\ \hline
\end{tabular}

\bigskip

\noindent
\textsc{Department of Mathematics, Statistics, and Computer Science (MC249)\\ }
\textsc{University of Illinois at Chicago\\ }
\textsc{851 South Morgan Street\\ }
\textsc{Chicago, IL  60680-7045\\ }
\textsc{USA\\}
{\em e-mail}: {\tt gtaylo9@uic.edu}\\
\\
\\
\textsc{Department of Mathematics \\}
\textsc{College of William and Mary \\}
\textsc{P. O. Box 8795 \\}
\textsc{Williamsburg, VA  23187-8795 \\}
\textsc{USA\\}
{\em e-mail}:  {\tt vinroot@math.wm.edu}\\
\end{document}